\documentclass[preprint,12pt]{elsarticle}

\usepackage{amssymb}
\usepackage{amsthm}
\usepackage{amsmath}

\usepackage[all]{xypic}
\usepackage[
	colorlinks=true, 
	linkcolor=blue,
	citecolor=blue
	]{hyperref}
\usepackage{breakcites}	
\usepackage{commath}

\journal{Journal of Differential Equations}

\newcommand{\C}{\mathcal{C}}  
\newcommand{\Ctheta}{\mathcal{C}_{\theta_{k}}} 
 
\newcommand{\thetak}{\theta_{k}}

\newcommand{\thetakp}{\theta_{k+1}}

\newcommand{\Jk}{J^k}
\newcommand{\Jkp}{J^{k+1}}
\newcommand{\Jkl}{J^{k-1}}
\newcommand{\Il}{I^{l}_{\theta_{k}}}
\newcommand{\Iln}{I^{l,n}_{\theta_{k}}}
\newcommand{\muL}[1]{U^{#1}_{\underline{L}}}
\newcommand{\Sk}{S^{k}\underline{R}^{*}\otimes N}
\newcommand{\Skp}{S^{k+1}\underline{R}^{*}\otimes N}
\newcommand{\Skl}{S^{k-1}\underline{R}^{*}\otimes N}
\newcommand{\Dub}{\mathcal{F}}

\DeclareMathOperator{\Hom}{Hom}

\newcommand{\pr}{\mathrm{\normalfont \bf{pr}}}

\usepackage{accents}
\newcommand{\pJ}{J^{k+1}_{\theta_{k},l}}
\newcommand{\pJL}{J^{k+1}_{\theta_{k},\underline{L}}}
\newcommand{\Dir}{\mathrm{\normalfont \bf{dir}}}

\newcommand{\E}{\mathcal{E}}

\newcommand{\V}{{\mathcal{V}}}
\newcommand{\F}{{\mathcal{F}}}
\newcommand{\uF}[1]{\underline{\mathcal{F}}^{#1}}

\newcommand{\Qphiq}{Q_{\phi_{q}}}

\newcommand{\phiq}{\phi_{q}}
\newcommand{\phiql}{\phi_{q-1}}

\newcommand{\de}{\dif }

\renewcommand{\P}{\mathcal{P}}
\newcommand{\uR}{{\underline{R}}}
\newcommand{\uL}{{\underline{L}}}
\newcommand{\utheta}{\theta_{0}}
 
\newcommand{\Gr}{\mathrm{Gr\,}}
\newcommand{\osc}{\mathrm{osc\,}}

\newcommand{\virg}[1]{``#1''}

\newcommand\restr[2]{{% we make the whole thing an ordinary symbol
  \left.\kern-\nulldelimiterspace % automatically resize the bar with \right
  #1 % the function
  \vphantom{\big|} % pretend it's a little taller at normal size
  \right|_{#2} % this is the delimiter
  }}

\newcommand{\St}{^{\textrm{st}}}

\newcommand{\Th}{^{\textrm{th}}}

\numberwithin{equation}{section}
\numberwithin{figure}{section}
    
\theoremstyle{plain}
\newtheorem{thm}{Theorem}
\newtheorem{prop}{Proposition}
\newtheorem{lem}{Lemma}
\newtheorem{cor}{Corollary}

\theoremstyle{definition}
\newtheorem{defn}{Definition}

\theoremstyle{remark}

\newtheorem{rem}{Remark}

\begin{document}

\begin{frontmatter}


%%\title{Title\tnoteref{label1}}
\ead{michael.baechtold@hslu.ch}

\title{Partial extensions of jets and the polar distribution on Grassmannians of non-maximal integral elements}


\author{Michael ~B\"achtold}

\address{Lucerne University of Applied Sciences and Arts\\
Technikumstrasse 21\\
CH--6048 Horw, Switzerland}

\begin{abstract}
We study an intrinsic distribution, called \emph{polar}, on the space of $l$-dimensional integral elements of the higher order contact structure on jet spaces. The main result establishes that this exterior differential system is the prolongation of a natural system of PDEs, named \emph{pasting conditions}, on sections of the bundle of \emph{partial jet extensions}. Informally, a partial jet extension is a $k\Th$ order jet with additional $k+1\St$ order information along $l$ of the $n$ possible directions. A choice of partial extensions of a jet into all possible $l$-directions satisfies the pasting conditions if the extensions coincide along pairwise intersecting $l$-directions.\par
We further show that prolonging the polar distribution once more yields the space of $(l,n)$-dimensional integral flags with its double fibration distribution. When $l>1$ the exterior differential system is holonomic, stabilizing after one further prolongation.
\par
The proof starts form the space of integral flags, constructing the tower of prolongations by reduction.
\end{abstract}

\begin{keyword}
Jet spaces \sep Exterior Differential Systems \sep Geometry of PDEs


\MSC[2010] 14M15 \sep % flags, grassmannians
35A99 \sep %PDEs, None of the above, but in this section
35A30 \sep  % PDE   -  Geometric theory, characteristics, transformations
57R99 \sep  %Manifolds and cell complexes, None of the above, but in this section
53B25 \sep  %Differential geometry, Local submanifolds
58A05 \sep  %	Differentiable manifolds, foundations
58A15 \sep %Exterior differential systems (Cartan theory)
58A17 \sep  %Pfaffian systems
58A20 \sep % jets
58A30 \sep %Vector distributions (subbundles of the tangent bundles)
58A32 \sep %Natural bundles
58C99 \sep % Calculus on manifolds; nonlinear operators
58J99 % Global Analysis on manifolds, None of the above, but in this section
\end{keyword}

\end{frontmatter}

\setcounter{tocdepth}{2}
\setcounter{section}{-1}
\tableofcontents

\section{Introduction}
Consider a manifold $J$ supplied with an exterior differential system and let $\theta \in J$ be a point. The space $I^{l}_{\theta}$ of $l$-dimensional regular integral elements of the exterior differential system at $\theta$ (we refer to \cite{MR1083148} for basic notions on exterior differential systems) is equipped with a natural distribution (in the sense of \emph{field of tangent planes}): a tangent vector at $L\in I^{l}_{\theta}$ belongs to this distribution if, considered as an infinitesimal first order motion of the integral element $L$, it leaves $L$ inside of its polar space. The existence of this distribution was pointed out to me by A. M. Vinogradov \cite{PolarPrivate} and was called the \emph{polar distribution} in \cite{BaechtMor}. \par
In the special case when the exterior differential system is the higher order contact structure (a.k.a. Cartan distribution) on the manifold  $J=J^{k}(2,1)$ of $k\Th$ order jets of functions in two independent and one dependent variables, the polar distribution on the space of one-dimensional horizontal integral was shown to be locally isomorphic to a Cartan distribution on the jet space of a \emph{new} bundle with \emph{one} independent and $k+1$ dependent variables \cite{ThesisMichi}. This result was extended in \cite{BaechtMor} to the case of one-dimensional integral elements of contact manifolds, i.e. the case  $J=J^{1}(n,1)$. Both proofs were in local coordinates and gave no hint on the geometrical origin of this new bundle, nor on ways to extend the result to other dimensions.\par
Here I remedy this  by:
\begin{itemize}
\item[a)] generalising the result to horizontal integral elements of arbitrary dimension $l<n$ in the Cartan distribution at a jet $\theta \in J^k(n,m)$, with an arbitrary number of independent and dependent variables $n$ resp. $m$. 
\item[b)] giving a coordinate free description of this new bundle and clarifying its geometrical meaning.
\end{itemize} 
The total space of the bundle mentioned in b) will be called the space of \emph{partial jet extensions of $\thetak$} and denoted by $\pJ$.  An element of $\pJ$ can be thought of as an extension of the $k\Th$ order jet $\thetak$ by $k+1\St$ order information in the direction of an $l$-dimensional subspace of the space of independent variables. In terms of local coordinates this means that an element of $\pJ$ specifies the values of the partial derivatives of order $k+1$, in $l$ chosen directions, in addition to the partial derivatives of order $\leq k$ determined by $\thetak$. The base of the bundle of partial jet extensions is the Grassmannian $\Gr(l,n)$ of all possible \virg{directions} along which to extend. Hence a section of the bundle $\pJ\to \Gr(l,n)$ specifies a partial extension of $\thetak$ along each $l$-dimensional direction. There is a natural condition for such a section to be \virg{holonomic}: when all extensions coincide on intersecting directions. These \emph{pasting conditions} can be reformulated as a system of first order linear PDEs on sections of $\pJ$. We then have
\begin{thm}[Main result, first part]\label{thm:mainfirstintro}
The polar distribution on $\Il$, considered as an exterior differential system, is the $k-1\St$ prolongation of the system of pasting conditions on the bundle of partial jet extensions.
\end{thm}
In the case of 1-dimensional integral elements the pasting conditions are trivially satisfied for any section, since there are no non-trivial intersections of one dimensional directions. Hence in that case, by theorem \ref{thm:mainfirstintro}, $\Il$ is the $k\Th$ jet space of the bundle $\pJ\to \Gr(l,n)$, in agreement with the results mentioned above \cite{ThesisMichi, BaechtMor}.\par
We further show 
\begin{thm}[Main result, continued]\label{thm:mainsecondintro}
Prolonging the polar distribution once more leads to the space $\Iln$ of $(l,n)$-dimensional integral flags with its canonical distribution induced from its double fibration structure. Finally, when $l>1$, prolonging once more stabilizes the process leading to an involutive distribution whose integral leaves are in one to one correspond with \virg{full} extensions of $\thetak$, i.e. jets of order $k+1$.
\end{thm}
The proofs of these theorems proceed from the top of the prolongation tower down: we consider the space of integral flags and exhibit certain natural distributions on it. From these we construct the tower of prolongations by reduction. Along the way we introduce moving frames adapted to the distributions which allow explicitly computations of commutation relations.\par
\subsection*{Structure of the article} 
The article consist of four sections. The first one gives detailed definitions and a precise statement of the main result, while the remaining three sections contain the proof. In section \ref{sect:constructing tower} the tower of fibrations $\Iln \to M^{k} \to \ldots \to M^{0} \to \Gr(l,n)$ is constructed and $M^{k}$ and $M^{0}$ are identified as $\Il$ and $\pJ$. In section \ref{sec:supplying with distr} we supply each manifold in the tower with a distribution and show that on $M^{k}$ this coincides with the polar distribution. In the final section \ref{sect:proof that tower is prolongations} we show that these distributions are consecutive prolongations and identify the one on $M^{1}$ with the pasting conditions.\par
\subsection*{Motivations and relations to other work} 
Spaces of lower dimensional integral elements in the Cartan distribution appear at several places in the theory of PDEs. They are central to characteristics, Monge cones, geometric singularities of PDEs \cite{MR923359, MR966202} and boundary conditions \cite{MorenoCauchy}. They have been used to find differential contact invariants of certain classes of PDEs \cite{MR2985508, DiffinvariantHyperbolicMonge}. Flags of integral elements appear in the context of the Cartan-K\"ahler theorem \cite{MR1083148}. To the authors knowledge, the polar distribution made its first appearance in the literature in \cite{ThesisMichi, BaechtMor}, where the reader may find simple applications to the classification of a third order PDE. The author is unaware of any previous appearance of the bundle of partial jet extension and the pasting conditions.\par
\section{Definitions and statement of main results}
\subsection{Conventions on jets}\label{subs:convent jets}
We work in the setting of jets of $n$-dimensional submanifolds in a fixed $n+m$-dimensional ambient manifold $E$. The space of $k\Th$ order jets is denoted with $\Jk =J^k(E,n)$. We think of such jets as \emph{infinitesimal germs} of $n$-submanifolds in $E$. The reader not familiar with jets of submanifolds might as well think of the locally isomorphic space of jets of sections of a bundle with $m$-dimensional fibers (corresponding to $m$ dependent variables) and $n$-dimensional base (corresponding to $n$ independent variables). We fix throughout a jet $\thetak \in \Jk$ of order $k\geq1$ and denote with $\Ctheta$ the plane of the Cartan distribution at $\thetak$. The terminology \emph{Cartan distribution} and \emph{higher contact structure} are used synonymously. Manifolds are real, although all arguments remain valid over any field of characteristic 0.\par
We shall make use of several facts concerning jets and the Cartan distribution which we collect in this subsection. The initiated reader may want to skip ahead to subsection \ref{subsec:integral elem} and return here when necessary. Further notational conventions may be found in section \ref{sec:conventions}.
\par
For $k>r$ there are natural projections $\pi_{k,r}:J^{k}\to J^{r}$  forgetting higher order information of jets. We say that $\thetak$ \emph{extends} the jet $\theta_{r}\in J^r$ (or that $\thetak$ \emph{restricts}  to $\theta_{r}$) when $\pi_{k,r}(\thetak)=\theta_{r}$. In particular, the restriction of $\thetak$ to order $0$ is a point in $E=J^{0}$ denoted with $\utheta$.\par
We use the convention of indexing the fiber of a bundle with its base point, and hence denote the manifold of all $k+1\St$ order jets extending $\thetak$ with $\Jkp_{\thetak}$. \par
There is a natural bijection between $n$-dimensional horizontal integral planes $R\subset\Ctheta$ and jets of order $k+1$ extending $\thetak$. Such integral planes are called R-planes in \cite{MR1670044}. The R-plane corresponding to $\thetakp\in \Jkp_{\thetak}$ is denoted with $R_{\thetakp}$. The R-plane in $J^0=E$ corresponding to the $1\St$ order restriction of $\thetak$ will be denoted with $\uR\subset T_{\theta_0}E$. \par
The fiber $\Jkp_{\thetak}$ is affine with underlying vector space $\Skp$ \cite{MR966202}, where $N$ is the \emph{normal tangent space} $T_{\theta_{0}}E/\uR$. We will interpret tensors in $\Skp$ as homogeneous polynomial maps from $\uR$ to $N$ of degree $k+1$.\par
For a distribution $\E$ on a manifold $M$ the \emph{curvature form} is the skew-symmetric tensor
\begin{equation}
\Omega: \E\wedge \E \to \left[\E,\E\right ]/\E
\end{equation}
induced by the Lie bracket of sections of $\E$. 
Here $\left[\mathcal{E},\mathcal{E}\right]$ denotes the \emph{derived distribution} of $\mathcal{E}$, which is the distribution spanned by $\mathcal{E}$ and Lie-brackets of fields in $\mathcal{E}$. The curvature form of the Cartan distribution $\C$ is called the \emph{metasymplectic form}. One can show that there is a natural isomorphism $\restr{\left[ \C,\C\right]/\C}{\thetak}\cong S^{k-1}\uR^*\otimes N$ so the metasymplectic form is considered of type
\begin{equation}
\Omega: \Ctheta\wedge \Ctheta\to \Skl .
\end{equation}
In standard local coordinates $x_i, u^j, u^j_\sigma$ on jet spaces $J^k$, where $x_i$ are the independent variables, $u^j$ are the dependent variables and $u^j_\sigma$ are jet coordinates with $\sigma=(\sigma_1,\ldots,\sigma_m)\in \mathbb{N}^m$ a multiindex of length $|\sigma|\leq k$, the metasymplectic structure acts as
\begin{align}
\Omega\left(D_{i},\,D_{j}\right) & = 0\\
\Omega \left(\partial_{u_{\sigma}^{j}},\,\partial_{u_{\sigma'}^{j'}}\right) & =  0\\
\Omega \left(\partial_{u_{\sigma}^{j}},\,D_{i}\right) & =  \partial_{u_{\sigma-1_{i}}^{j}}\label{metasymplecticnontric} .
\end{align}
Here $D_{i}=\partial_{x_i}+\sum_{|\sigma|<k}u_{\sigma+1_i}^{j}\partial_{u_{\sigma}^{j}}$ are \emph{total derivatives} and the vertical fields $\partial_{u_{\sigma}^{j}}$ correspond to the homogeneous polynomials 
\begin{equation}
\frac{1}{\sigma!}(dx_{1})^{\sigma_{1}}\cdot\ldots\cdot(dx_{n})^{\sigma_{n}}\otimes\frac{\partial}{\partial u^{i}}\in \Skl
\end{equation}
under the identification $\restr{\left[ \C,\C\right]/\C}{\thetak}\cong S^{k-1}\uR^*\otimes N$.\par
\subsection{Integral elements and the polar distribution}\label{subsec:integral elem}
Recall that a vector subspace $L\subset \Ctheta$ is called an \emph{integral element} \cite{MR1083148} (or involutive subspace in \cite{MR1670044}) of the Cartan distribution, if all differential forms in the differential ideal generated by the Cartan distribution vanish when restricted to $L$. Equivalently, $L$ is integral if the metasymplectic form $\Omega$ vanishes when restricted to $L$. Such a plane is \emph{horizontal} if it is transversal to fibers of the projection $\pi_{k,k-1}:\Jk\to \Jkl$, which turns out to imply transversality with respect to $\pi_{k,0}:\Jk\to J^0$.
\begin{defn}
The space of \emph{horizontal $l$-dimensional integral elements of $\Ctheta$} is
\begin{equation}
\Il:=\left \{L\subset \Ctheta \,\middle|\, \dim L=l, \ \Omega|_L= 0,\  L\textrm{ horizontal}\right \} .
\end{equation}
\end{defn}
Horizontal integral elements of maximal dimension are precisely R-planes \cite{MR1670044}.\par
To define the polar distribution on $\Il$ recall that the \emph{polar space}  $L^\perp $ \\\cite{MR1083148}  of an integral element $L$ is defined as the $\Omega$-orthogonal of $L$:
\begin{equation}
L^{\perp}:=\left \{ v\in \Ctheta \,\middle|\, \Omega(v,w)=0\; \forall w\in L \right \}  .
\end{equation}
Since there is a canonical identification
\begin{equation}
T_L\Gr(\Ctheta,l)\cong \Hom(L,\Ctheta / L )
\end{equation}
\cite{MR1416564} and since $\Il \subset \Gr(\Ctheta,l)$, a tangent vector $\dot{L}$ at $L\in \Il$ may be understood as a linear map
\begin{equation}\label{tangentvectortoGr}
L\stackrel{\dot{L}}{\longrightarrow}\frac{\Ctheta}{L}  .
\end{equation}
We define the \emph{osculator}  of $\dot{L}$ as 
\begin{equation}
\osc \dot{L} := \text{pr}^{-1}\operatorname{im} \dot{L}
\end{equation}
where $\text{pr}:\Ctheta\to \frac{\Ctheta}{L}$ is the canonical projection and $\operatorname{im} \dot{L}$ is the image of $\dot{L}$. The osculator may be thought of as the span of $L$ and all infinitesimally near $L_{t}\in \Gr(\Ctheta,l)$ reached by the infinitesimal displacement $\dot {L}$.\par
Using these notions we give
\begin{defn}
The plane of the \emph{polar distribution} $\P$ on $\Il$ at $L\in \Il$ is
 \begin{equation}
\P_L:=\left\{\dot{L}\in T_L (\Il) \,\middle|\,  \osc\dot{L}\subseteq  L^\perp  \right \}.
\end{equation}
\end{defn}
Alternatively we have the simpler but equivalent description
\begin{equation}\label{eq:descrpolar}
\P_{L}=\left \{\dot{L}\in \Hom(L, \Ctheta / L) \,\middle|\, \Omega(l_{1},\dot{L}(l_{2}))= 0 \textrm{ for all }l_{1},l_{2}\in L \right \},
\end{equation}
where we interpret $\dot{L}$ as the map \ref{tangentvectortoGr} and $\dot{L}(l_{2})$ is its application to $l_{2}$.
Description \ref{eq:descrpolar}  follows from
\begin{lem} A vector $\dot{L}\in T_L\Gr(\Ctheta,l)$ is tangent to the submanifold $\Il\subset \Gr(\Ctheta,l)$ at $L\in\Il$ iff
\begin{equation}\label{eq:tangentspaceIl}
\Omega(l_{1},\dot{L}(l_{2}))= \Omega(l_{2},\dot{L}(l_{1}))  \quad \textrm{for all} \quad l_{1},l_{2}\in L .
\end{equation}
\end{lem}
\begin{proof}
Given $\dot{L}\in T_{L}\Il$ we first show that \ref{eq:tangentspaceIl} holds. For this consider a smooth one parameter family of planes $L_{t}\in \Il$ with $L_{0}=L$, $\restr{\frac{\de }{\de t}L_{t}}{t=0}=\dot{L} $ and families of vectors $l_{i}(t)\in L_{t}$ with $l_{i}(0)=l_{i}$ and 
\begin{equation}\label{eq:parametervectors}
\left( \restr{\frac{\de }{\de t}l_{i}(t) }{t=0}\mod L \right)=\dot{L}(l_{i})
\end{equation}
for $i=1,2$. Since $L_{t}$ is integral we have
\begin{equation}\label{eq:omega}
\Omega \left(l_{1}(t),l_{2}(t)\right)=0 .
\end{equation}
Taking the derivative with respect to $t$ at $t=0$ on both sides of equation \ref{eq:omega} using the product rule and \ref{eq:parametervectors} we obtain
\begin{equation}
\Omega\left(l_{1},\dot{L}(l_{2})\right)+\Omega\left(\dot{L}(l_{1}),l_{2}\right)=0 ,
\end{equation}
which by skew-symmetry of $\Omega$ leads to \ref{eq:tangentspaceIl}.
\par
Conversely, assume $\dot{L}\in \Hom(L,\Ctheta/L)$ satisfies \ref{eq:tangentspaceIl} with $L\in \Il$. We need to show $\dot{L}\in T_{L}\Il$. Choose an R-plane $R$ such that $L\subset R$ and choose a splitting $R=L\oplus L^{\text{comp}}$. This gives a splitting of the Cartan plane into three components $\Ctheta=L\oplus L^{\text{comp}} \oplus (S^{k}R^{*}\otimes N)$ where the last component is the tangent space to the fiber of $\pi_{k,k-1}:J^{k}\to J^{k-1}$. 
This induces a decomposition $\dot{L}=\dot{L}_{\text{vert}}\oplus \dot{L}_{\text{hor}}$ into a vertical $\dot{L}_{\text{vert}}:L\to L^{\text{comp}}$ and horizontal component $\dot{L}_{\text{hor}}:L\to S^{k}R^{*}\otimes N$. 
Substituting $\dot{L}$ with $\dot{L}_{\text{vert}}\oplus \dot{L}_{\text{hor}}$ in \ref{eq:tangentspaceIl} we find that $\dot{L}_{\text{vert}}$ satisfies $\Omega\left(l_{1},\dot{L}_{\text{vert}}(l_{2})\right)+\Omega\left(\dot{L}_{\text{vert}}(l_{1}),l_{2}\right)=0$. This implies that the graph of $\dot{L}_{\text{vert}}$ is an $l$-dimensional integral element in $\Ctheta$. 
Pick an R-plane $R'$ such that $\text{graph}(\dot{L}_{\text{vert}})\subset R'$ and interpret this new R-plane as the graph of a linear map $A:R\to S^{k}R^{*}\otimes N$. Since $R'$ is integral it follows that  $\Omega\left(r_{1},A(r_{2})\right)+\Omega\left(A(r_{1}),r_{2}\right)=0$ for all $r_{i}\in R$. Moreover by construction $A(l)=\dot{L}_{\text{vert}}(l)$ for all $l\in L$.\par
We now define a one parameter family of $l$-dim planes $L_{t}\in \Gr(\Ctheta,l)$ as follows. Pick a basis $b_{1},\ldots,b_{l}$ of $L$ and define $L_{t}$ as the span of the vectors $b_{i}(t):=b_{i}+t\dot{L}_{\text{hor}}(b_{i}) + t\cdot A(b_{i}+t\dot{L}_{\text{hor}}(b_{i}))$. It is easy to see that $L_{0}=L$ and $\restr{\frac{\de }{\de t}L_{t}}{t=0}=\dot{L} $. We claim that all the $L_{t}$ are integral elements, which would finish the proof. For this it suffices to show that $\Omega(b_{i}(t),b_{j}(t))=0$ wich follows from a straightforward computation.
\end{proof}

\subsection{The bundle of partial jet extensions}
In this subsection we define the space of partial jet extensions of a $k\Th$ order jet $\thetak$, together with its fibration over a standard Grassmannian. We do this by introducing an equivalence relation on all $k+1\St$ order jets extending $\thetak$.

Let $\thetakp, \thetakp'\in\Jkp_{\thetak}$ be two $k+1\St$ order jets extending $\thetak$ and fix $\uL\in \Gr(\uR,l)$. We think of $\thetakp, \thetakp'$ as infinitesimal germs of submanifolds in $E$ having contact of order $k$ along $\thetak$, while $\uL$ is thought of as an $l$-dimensional direction inside these germs. 

Using local coordinates the equivalence relation is defined as follows: choose splitting coordinates $x_{1},\ldots,x_{n},u_{1},\ldots,u_{m}$ on $E$ centered at $\theta_{0}$  such that $\uL$ is spanned by $\partial_{x_{1}},\ldots,\partial_{x_{l}}$. Let
\begin{equation}\label{localsectionjet1}
u_{j}=F_{j}(x_{1},\ldots,x_{n})
\end{equation}
and
\begin{equation}\label{localsectionjet2}
u_{j}=G_{j}(x_{1},\ldots,x_{n})
\end{equation} 
be two sets of locally defined functions with $j=1,\ldots,m$, such that $\thetakp$ (resp. $\thetakp'$) is the $k+1\St$ jet of \ref{localsectionjet1} (resp. \ref{localsectionjet2}). Hence the jets  $\thetakp$ and $\thetakp'$ are determined by all partial derivatives of $F$ and $G$ at $0$ of order $\leq k+1$. Since $\thetakp$ and $\thetakp'$ are tangent of order $k$, all partial derivatives of $F$ and $G$ at $0$ of order $\leq k$ are equal. 

\begin{defn}
We say that $\thetakp$ and $\thetakp'$ are  \emph{tangent of order $k+1$ in direction $\uL$} if all $k+1\St$ order partial derivatives of $F$ and $G$ involving \emph{only} $\partial_{x_{1}},\ldots,\partial_{x_{l}}$ coincide. 
\end{defn}
An equivalent coordinate independent description is given by
\begin{lem}\label{lem:tangencyordk+1}
Two jets $\thetakp$ and $\thetakp'$ extending $\thetak$ are tangent of order $k+1$ along $\uL$ iff the polynomial $\thetakp-\thetakp'\in \Skp$ vanishes when restricted to $\uL$.
\end{lem}
\begin{proof}
From the properties of the affine $\Skp$-structure on $\Jkp_{\thetak}$. See for instance \cite{MR966202}.
\end{proof}
It follows immediately that tangency of order $k+1$ along $\uL$ is an equivalence relation on jets of order $k+1$ extending $\thetak$, which leads to
\begin{defn} We denote with $\pJL$ the quotient set under this equivalence relation and call it the \emph{space of partial extensions of $\thetak$ along $\uL$}. 
\end{defn}
We think of an element in $\pJL$ as the jet $\thetak$ with additional $k+1\St$ order information in direction of $\uL$.  
%TODO Picture
\par
By varying $\uL$ in $\Gr(\uR,l)$, the spaces $\pJL$ make up the fibers of a bundle which we denote with
\begin{align}
\Dir: \pJ & \rightarrow \Gr(\uR,l)\\
        \phi &\mapsto \uL=\Dir(\phi)
\end{align}
where the projection $\Dir$ maps a partial extension $\phi$ to its direction of extension $\uL$. By definition, a section of the bundle $\Dir$ specifies a partial jet extension of $\thetak$ along every direction $\uL\in \Gr(\uR,l)$.\par
There is a natural \virg{holonomicity} condition for such a section.
\begin{defn}\label{defn:pasting}
We say a section $s:\Gr(\uR,l) \to \pJ$ of partial jet extensions satisfies the \emph{pasting conditions} (or is \emph{holonomic}), if for any two directions $\uL,\uL' \in  \Gr(\uR,l)$ the partial extensions $s(\uL),s(\uL')$ coincide on the intersection $\uL\cap \uL'$. This means that, jets $\thetakp$ and $\thetakp'$ representing $s(\uL)$ resp. $s(\uL')$ have contact of order $k+1$ along $\uL\cap \uL$.
\end{defn} 
Using lemma \ref{lem:tangencyordk+1} it is straightforward to check that this definition is independent of the choice of representatives 
$\thetakp$ and $\thetakp'$. We call these the pasting conditions since they express when a section of partial extensions can be \virg{glued together} to form a full $k+1\St$-oder extension of $\thetak$. This last statement will actually be a consequence of the main result. 
\subsection{Infinitesimal pasting conditions} 
\label{subsec:infinitesimalpasting}
The pasting conditions can be reformulated as a system of $1\St$ order PDEs on sections of $\Dir$, which we call the \emph{infinitesimal pasting conditions}.  To write down this system of PDEs we introduce local coordinates that shall be used throughout the rest of this article.\par
On the base space $\Gr(\uR,l)$ we choose standard affine coordinates on Grassmannians: fix an element $\uL_0 \in \Gr(\uR,l)$, choose a basis 
\begin{equation}\label{eq:y's}
y_1,\ldots,y_d\in \uL_0^\circ
\end{equation}
of the annihilator $\uL_0^\circ$  and complement it to a basis of $\uR^*$ with covectors 
\begin{equation}\label{eq:x's}
x_1,\ldots,x_l\in \uR^*.
\end{equation}
Then for any plane $\uL\in  \Gr(\uR,l)$ transversal to the complement
\begin{equation}
\uL_{0}^{\text{compl}}:=\bigcap_{1\leq j \leq l} \ker{x_j}
\end{equation}
there are unique coefficients $A_{i,j}$ such that the covectors 
\begin{equation}
y_i- \sum_{j=1}^l A_{i,j}\,x_j \quad \textrm{ with }   \quad i=1,\ldots,d
\end{equation}
form a basis of the annihilator of $\uL$ and conversely any such choice of coefficients determines  such a plane. Hence the $A_{i,j}$ serve as local coordinates on $\Gr(\uR,l)$.\par
Coordinates on fibers: By lemma \ref{lem:tangencyordk+1} each fiber $\pJL$ of $\Dir$ is an affine quotient of $\Jkp_{\thetak}$ with underlying vector space $S^{k+1}\uL^{*}\otimes N$. Since $\Jkp_{\thetak}$ is affine over $\Skp$ we fix a jet
\begin{equation}\label{eq:basejet}
\theta_{k+1,\text{orig}} \in \Jkp_{\thetak}
\end{equation} 
as the \virg{origin} and identify $\Jkp_{\thetak}$ with the vector space $\Skp$ and hence $\pJL$ with $S^{k+1}\uL^{*}\otimes N$. Since each $\uL$ transversal to $\uL_{0}^{\text{compl}}$ is further identified with $\uL_{0}$ by the natural projection $\uR=\uL_{0}\oplus \uL_{0}^{\text{compl}}\to \uL_{0}$, we obtain an identification of  $S^{k+1}\uL^{*}\otimes N$ with $S^{k+1}\uL_{0}^{*}\otimes N$ and hence an identification $\pJL\cong S^{k+1}\uL_{0}^{*}\otimes N$. So choosing a basis 
\begin{equation}\label{eq:e's}
e_1,\ldots,e_m
\end{equation} 
of $N$, each point in $\pJ$ above our chart on $\Gr(\uR,l)$ is specified by its base coordinates $A_{i,j}$ plus the coefficients $v_{\lambda}^{h}$ of a homogeneous polynomial
\begin{equation}
\sum_{\lambda, h}v_{\lambda}^{h}\,x^{\lambda}\otimes e_{h}
\end{equation}
where $\lambda=(\lambda_{1},\ldots,\lambda_{l})\in\mathbb{N}^{l}$ denotes a multiindex of length $|\lambda|=k+1$ and $x^{\lambda}=x_{1}^{\lambda_{1}}\cdot \ldots \cdot x_{l}^{\lambda_{l}}$. \par
With these local coordinates
\begin{equation}\label{eq:localcoordinates first}
A_{i,j},\;v_{\lambda}^{h}
\end{equation}
on the total space of $\Dir$, a local section of $\Dir$ is given by functions 
\begin{equation}
v_{\lambda}^{h}\left( A \right)
\end{equation} 
where $A$ is short for all the variables $A_{i,j}$. Such a section satisfies the non-infinitesimal pasting conditions from definition \ref{defn:pasting} iff, for any two planes $\uL,\uL'\in \Gr(\uR,l)$ with coordinates $A, A'$ we have
\begin{equation}\label{globalpasting1}
\sum_{\lambda, h}v_{\lambda}^{h}(A)\, x^{\lambda}\otimes e_{h}=\sum_{\lambda, h}v_{\lambda}^{h}(A')\,x^{\lambda}\otimes e_{h}
\end{equation} 
whenever $x=(x_{1},\ldots, x_{l})$ satisfies
\begin{equation}\label{globalpasting2}
\sum_{j} A_{i,j}\, x_{j}= \sum_{j}A'_{i,j}\, x_{j}
\end{equation}
for all $i=1,\ldots,d$. \par
To derive the infinitesimal pasting conditions from \ref{globalpasting1}, \ref{globalpasting2} we fix $\uL\in \Gr(\uR,l)$ with coordinates $A$ and consider two continuos perturbations of $\uL$: one perturbation changing entry $A_{i,j}$ of matrix $A$ to  $A_{i,j}+t$, with $t$ a perturbation parameter, and leaving the other entries fixed. The other perturbation changing entry $A_{i,j'}$ to $A_{i,j'}+s$ with parameter $s$ and leaving all other entries unperturbed. Here $i,j,j'$ are fixed indices. We write the perturbed matrices as
\begin{align}
A+t1_{i,j} \label{perturbation with t}\\ 
A+s1_{i,j'}. \label{perturbation with s}
\end{align}
For a section of $\Dir$ that satisfies the pasting conditions, we substitute $A$ with  \ref{perturbation with t} and $A'$ with  \ref{perturbation with s} in \ref{globalpasting1} to obtain
\begin{equation}\label{globalpasting1special}
\sum_{\lambda}v_{\lambda}^{h}(A+t1_{i,j})\,x^{\lambda}\otimes e_{h}=\sum_{\lambda}v_{\lambda}^{h}(A+s1_{i,j'})\,x^{\lambda}\otimes e_{h}
\end{equation} 
whenever $x=(x_{1},\ldots, x_{l})$ satisfies
\begin{equation}\label{globalpasting2special}
t x_{j}= s x_{j'}
\end{equation}
according to \ref{globalpasting2}.
Taking the total differential of both sides of equations \ref{globalpasting1special} and \ref{globalpasting2special}  (where the variables are $s,t,x$ while $A$ is assumed fixed, i.e. $dA=0$)  we obtain
\begin{multline}\label{globalpasting1specialder}
\sum_{\lambda}\partial_{A_{i,j}}v_{\lambda}^{h}(A +t1_{i,j})\, x^{\lambda}\de t\otimes e_{h} + \sum_{\lambda, \iota}v_{\lambda}^{h}( A+t1_{i,j})\,\lambda_{\iota}x^{\lambda-1_{\iota}}\de x_{\iota}\otimes e_{h}=\\
\sum_{\lambda}\partial_{A_{i,j'}}v_{\lambda}^{h}\left( A+s1_{i,j'} \right)x^{\lambda}\de s\otimes e_{h} + \sum_{\lambda, \iota}v_{\lambda}^{h}\left( A+s1_{i,j'} \right)\lambda_{\iota}x^{\lambda-1_{\iota}}\de x_{\iota}\otimes e_{h}
\end{multline} 
from \ref{globalpasting1special}, while from \ref{globalpasting2special} we obtain
\begin{equation}\label{globalpasting2specialder}
x_{j}\de t+t\de x_{j}= x_{j'}\de s+s\de x_{j'}.
\end{equation}
Now set $t=s=0$, so \ref{globalpasting1special} \ref{globalpasting2special} are trivially satisfied while \ref{globalpasting1specialder} becomes
\begin{equation}\label{globalpasting1specialder_st0}
\sum_{\lambda}\partial_{A_{i,j}}v_{\lambda}^{h}\left( A\right)x^{\lambda}\de t \otimes e_{h} =
\sum_{\lambda}\partial_{A_{i,j'}}v_{\lambda}^{h}\left( A \right)x^{\lambda}\de s \otimes e_{h}
\end{equation}
after canceling equal terms. Equation \ref{globalpasting2specialder} becomes
\begin{equation}\label{globalpasting2specialder_st0}
x_{j}\de t= x_{j'}\de s.
\end{equation}
We may multiply both sides of equation \ref{globalpasting1specialder_st0} with $x_{j}$ and substitute $x_{j}\de t$ with $x_{j'}\de s$ by \ref{globalpasting2specialder_st0}  to find
\begin{equation}\label{infintesimalpasting}
\sum_{\lambda}\partial_{A_{i,j}}v_{\lambda}^{h}\left( A\right)x^{\lambda+1_{j'}}\otimes e_{h}=
\sum_{\lambda}\partial_{A_{i,j'}}v_{\lambda}^{h}\left( A \right)x^{\lambda+1_{j}}\otimes e_{h}
\end{equation}
where we have canceled $\de s$.
Since equations \ref{infintesimalpasting} hold for arbitrary values of $x$ we can equate the coefficients on both sides to find that, in local coordinates, a section $v_{\lambda}^{h}\left( A \right)$ that satisfies pasting conditions \ref{globalpasting1}, \ref{globalpasting2}, also satisfies the \emph{infinitesimal pasting conditions}
\begin{align}
\partial_{A_{i,j}}v_{\lambda}^{h}&=\partial_{A_{i,j'}}v_{\lambda'}^{h}\quad \text{whenever} \quad \lambda-1_{j}=\lambda'-1_{j'} ,\label{infpaste1}\\
\partial_{A_{i,j}}v_{\lambda}^{h}&=0 \quad \text{whenever} \quad \lambda_{j}=0  \label{infpaste2} .
\end{align}
Observe that when $l=1$ these conditions are trivially satisfied, so the equations are \virg{empty}.
\begin{rem}
If one considers perturbations $A+t1_{i,j}$, $A+s1_{i',j'}$ with different indices $i\neq i'$ one finds again equations \ref{infpaste2}. In fact, the prolongation theorems \ref{thm:mainfirstintro} and \ref{thm:mainsecondintro} will establish that all possible differential consequences of the non-infinitesimal pasting conditions \ref{globalpasting1}, \ref{globalpasting2} coincide with the differential consequences of the infinitesimal pasting conditions \ref{infpaste1}, \ref{infpaste2}.
\end{rem}
\subsection{Integral flags and the double fibration structure}
A further ingredient of the main theorem is the space of partial flags of integral elements which we introduce here. This space is also the starting point for constructing the tower of prolongations and thereby for proving the main theorem.
\begin{defn}
A pair $(L,R)$ of subspaces $L\subset R\subset \Ctheta$ with $R$ an R-plane and $L$ of dimension $l$ will be called a \emph{$(l,n)$-dimensional flag of horizontal integral elements}. The space of all such integral flags is denoted with
\begin{equation}
\Iln:=\left \{(L,R)\,\middle|\, L\in \Il,\; R\in I^n_{\thetak},\; L \subset R\right \} .
\end{equation}

\begin{rem}
By established terminology it would be correct to call these \emph{partial} flags. We omit the adjective to simplify the terminology.
\end{rem}

\end{defn}
The space of integral flags is naturally fibered in two ways: one projection forgets the smaller integral element $L$ and remebers only $R$. Since $R$ is an R-plane corresponding to some jet of order $k+1$ we write this projection as:
\begin{align}
\pr_{n}:\Iln &\to \Jkp_{\thetak} \\
 (L, R_{\thetakp}) & \mapsto \thetakp .
 \end{align}
The second projection forgets $R$ and is hence of the form
\begin{align}
\pr_{l}:\Iln &\to \Il \\
 (L,R_{\thetakp})&\mapsto L .
\end{align}
We picture both of these as a double fibration
\begin{equation}
\xymatrix{
& \Iln \ar[dr]^{\pr_{l}}\ar[dl]_{\pr_{n}}& \\
\Jkp_{\thetak} && \Il .
}
\end{equation}
This double fibration gives rise to a natural distribution on $\Iln$. Recall that for a fiber bundle $\pi:A\to B$ the \emph{vertical distribution} $V\pi$ on the total space $A$ consists of all vectors tangent to the fibers.
\begin{defn}
The sum of the two vertical distributions associated to the projections $\pr_{l}$ and $\pr_{n}$ defines a distribution 
\begin{equation}
\Dub:=V\pr_{l}+V\pr_{n} 
\end{equation} 
on $\Iln$ which we call the \emph{flag distribution} on the space of integral flags.
\end{defn}
\subsection{Statement of the main results}
Before we state the main result we recall the notion of \emph{prolongation of an exterior differential system with independence conditions}. We shall only need the case where the independence conditions are given by transversality conditions with respect to a bundle projection $\pi:M\to N$, and where the exterior differential system on $M$ is a distribution $\mathcal{E}$ (i.e. a Pfaffian system). We refer to \cite{MR1083148} for the general definition.\par
To proceed we need to recall the notions of relative distribution and lift of a (relative) distribution.
\begin{defn}\label{defn:relative distr lift}
Given a fiber bundle $f:A\to B$, a \emph{relative distribution $\mathcal{D}$ along $f$} is vector sub-bundle of the pullback of the tangent bundle $TB$ to $A$. In other words, a relative distribution attaches to every point $a\in A$ a tangent plane $\mathcal{D}_{a}\subset T_{f(a)}B$ in a smooth way.  Any relative distribution $\mathcal{D}$ along $f$, can be \emph{lifted} to a non-relative distribution $f^{-1}\mathcal{D}$ on $A$  by defining $(f^{-1}\mathcal{D})_{a}:=(Tf)^{-1}(\mathcal{D}_{a})$. 
\end{defn}
\begin{rem}
Lifting relative distributions induces a canonical correspondence between relative distributions along $f$ and distributions on $A$ containing the vertical distribution $Vf$. Note also that every non-relative distribution on $B$ can be seen as a relative distribution along $f$ in an obvious way. 
\end{rem}
Returning to the notion of prolongation of a distribution $(M,\mathcal{E})$, one defines the manifold $M^{(1)}$ to consist of all $(\dim N)$-dimensional $\pi$-horizontal integral elements of $\mathcal{E}$. The prolonged distribution $\mathcal{E}^{(1)}$ on $M^{(1)}$ is then defined to be the lift of the \emph{tautological relative distribution} along the natural projection $\pi^{(1)}:M^{(1)}\to M$. The tautological relative distribution by definition attaches to each $S\in M^{(1)}$ the subspace $S\subset T_{\pi^{(1)}S}M$. Since $M^{(1)}$ is still a bundle over $N$ via $\pi\circ \pi^{(1)}$ we can iterate this construction an define the second prolongation etc.\par
\begin{thm}[Main theorem]\label{thm:main}
The $k-1\St$ prolongation of the system of infinitesimal pasting conditions is the polar distribution on $\Il$. The $k\Th$ prolongation is the space of integral flags with its flag distribution. Moreover, when $l>1$ the $k+1\St$ prolongation is an involutive distribution whose maximal integral submanifolds are in one-to-one correspondence with jets of order $k+1$ prolonging $\thetak$. When $l=1$ the pasting conditions are empty and so $\Il=J^{k}(\Dir)$ and $\Iln=J^{k+1}(\Dir)$ while the polar and flag distributions are the Cartan distributions on $J^{k}(\Dir)$ resp. $J^{k+1}(\Dir)$.
\end{thm}
\begin{rem}
In local coordinates an exterior differential system is a system of PDE's, while prolonging amounts to taking total derivatives of the equations and adding them to the system. For this reason we occasionally refer to prolongation as \virg{adding differential consequences} to a system of PDEs.
\end{rem}
\section{Constructing the tower of fibrations}\label{sect:constructing tower}
In this section we exhibit a natural chain of involutive distributions 
$
\V^{0}\subset \V^{1}\subset \ldots \subset \V^{k+2}
$
on the space of integral flags $\Iln$. Their leaf spaces then yield the tower of fiber bundles 
\begin{equation}\label{eq: tower of prolongations 1}
\Iln \to M^{k} \to \ldots \to M^0 \to M^{-1}.
\end{equation}
Having done that, we recognize $M^{k}$ as $\Il$, $M^{0}$ as $\pJ$ and $M^{-1}$ as $\Gr(\uR,l)$. In section \ref{sec:supplying with distr} we then show how each $M^{q}$, $q>1$ is equipped with a natural distribution. 
\subsection{Internal structure of the tangent space $T\Iln$}
Since any integral element $L\in \Il$ is transversal to $\pi_{k,0}:J^{k}\to J^{0}$ we may project it down to $\uR\subset T_{\theta_{0}}E$ to obtain a subspace we denote with $\uL \in \Gr(\uR,l)$ (This projection also induces a canonical isomorphism $L\cong \uL$ which we shall use implicitly). Hence $\Il$ is naturally fibered over $\Gr(\uR,l)$:
\begin{align}
\Il & \to \Gr(\uR,l) \\
L & \mapsto \uL=T\pi_{k,0}(L ).
\end{align}
Using this projection we note the following important decomposition of the space of integral flags.
\begin{lem}\label{lem:product_decomp_flags} 
The map 
\begin{align}
\Iln & \to \Gr(\uR,l)\times \Jkp_{\thetak}\\
(L,R_{\thetakp}) & \mapsto (\uL,\thetakp)
\end{align}
is a canonical diffeomorphism of manifolds.
\end{lem}
\begin{proof} 
The inverse can be described by 
\begin{equation}
(\uL,\thetakp)\mapsto \left( R_{\thetakp}\cap\left(T_{\thetak}\pi_{k,0}\right)^{-1}(\uL)\, , \,R_{\thetakp}\right).
\end{equation}
\end{proof}
We henceforth use this identification $\Iln = \Gr(\uR,l)\times \Jkp_{\thetak}$ without explicit mention. 
It immediately exposes the following \virg{internal structure} on tangent spaces of $\Iln$.
\begin{cor}
The tangent space $T_{(L,R)}\Iln$ at $(L,R)\in \Iln$ is canonically isomorphic to 
\begin{equation}
\Hom( \uL, \uR/\uL ) \oplus \left( \Skp \right).
\end{equation} 
\end{cor}

\begin{proof}
The two summands correspond precisely to the tangent spaces of the components $\Gr(\uR,l)\times \Jkp_{\thetak}$.
\end{proof}

A generic vector in $\Hom( \uL, \uR/\uL ) \oplus \left( \Skp \right)$ will henceforth be denoted with $h\oplus f$, where $h\in \Hom( \uL, \uR/\uL )$ and $f\in \Skp$.

%For future use we also record the following simple
%\begin{lem} The composition of the natural projections $\Iln\to \Il $ with $\Il \to \Gr (\underline{R},l)$ is the projection onto the first factor in the decomposition $\Iln = \Gr(\underline{R},l)\times J^{k}_{\theta_{k-1}}$.
%\end{lem}
%
%\begin{proof}
%Immediate.
%\end{proof}
%

\subsubsection{A filtration on homogeneous polynomials}
The subspace $\uL\subset \uR$ associated to an integral flag $(L,R)$ gives rise to a filtration on the second component $\Skp$ of the tangent space of $\Iln$ at $(L,R)$:
\begin{defn}\label{defn:filtr_poly}
For $p=0,1,\ldots,k+2$ define 
$
\muL{p}
$
to be the vector subspace of $\Skp$ consisting of all homogeneous polynomials that vanish after taking $p$ derivatives in direction of $\uL$.  Equivalently, $\muL{p}$ consists of all symmetric $k+1$-multilinear forms on $\uR$ that vanish when inserting $p$ elements of $\uL$. 
\end{defn}
These subspaces form a natural filtration in $\Skp$ depending on $\uL\in \Gr(\uR,l)$:
\begin{equation}\label{eq: tower of filtrations on polys}
\underbrace{\muL{0}}_{=0} \subset \muL{1} \subset \ldots \subset \muL{k+1} \subset \underbrace{\muL{k+2}}_{=\Skp} .
\end{equation}
A basis of $\muL{p}$ may be constructed as follows: fix a basis $y_1,\ldots,y_d$ of $\uL^\circ$ and complement it with forms $x_1,\ldots,x_l$ to a basis of $\uR^*$. Denote symmetric monomials of these basic forms with
\begin{equation}
y^\delta x^\lambda := y_1^{\delta_1}\cdots y_d^{\delta_d} x_1^{\lambda_1}\cdots x_l^{\lambda_l}
\end{equation}
where $\delta=(\delta_1,\ldots,\delta_d)\in \mathbb{N}^d$ and $\lambda=(\lambda_1,\ldots,\lambda_l)\in \mathbb{N}^l$ are multi indices. Let $e_1,\ldots,e_m$ be the basis \ref{eq:e's} of $N$.

\begin{lem}\label{lem:basis_filt_poly}
$\muL{p}$ is generated by all symmetric tensors $ y^\delta x^\lambda \otimes e_h$ which are of degree less than $p$ in the $x$'s. More formally 
\begin{equation}
\muL{p}=\left \langle  y^\delta x^\lambda \otimes e_h  \,\middle|\,  \left | \delta \right |+ \left | \lambda \right | = k + 1 ,  | \lambda  | < p, h=1,\ldots, m \right \rangle.
\end{equation}
%We adopt the convention that the zero polynomial is of any degree.
\end{lem}

\begin{proof}
This follows straightforwardly from interpreting such symmetric tensors as polynomial maps. 
\end{proof}

Denoting with 
\begin{equation}
q:=k+1-p
\end{equation} the \emph{complementary degree} to $p$, we may describe $\muL{p}$ as all symmetric tensors of degree at least $q$ in the $y$'s.

\begin{cor}
\begin{align}
\muL{1}&=S^{k+1} \uL^\circ\otimes N \\
\muL{k+1}&= \text{polynomials vanishing on } \uL
\end{align}
\end{cor}
%
%The following lemma describes the associated graded to the filtration \txr{Needed?}.
%%
%\begin{lem}
%The component $\muL{p+1}/\muL{p}$ of the associated graded space is naturally isomorphic to \txr{fix if needed} $S^{k+1-q}\uL^*\otimes S^{q}\uL^\circ$ via the map which sends a $k+1$ multilinear form $\varphi(\cdot,\ldots,\cdot)$ to $(l_1,l_2,\ldots,l_q)\mapsto  \varphi(l_1,l_2,\ldots,l_q,\cdot,\ldots,\cdot)$ where $l_1,l_2,\ldots,l_q$ are vectors in $\underline{L}$.
%\end{lem}
%%
%
\subsection{Higher vertical distributions on $\Iln$}
The filtration \ref{eq: tower of filtrations on polys} of $\Skp$ from the previous subsection induces a natural chain of distributions on the tangent spaces of $\Iln$. From these we will construct the tower of prolongations \ref{eq: tower of prolongations 1}.%
\begin{defn}\label{GenVertical}
For $p=0,\ldots,k+2$ define the $p\Th$ \emph{vertical distribution} $\V^p$ on $\Iln$ at a point $(L,R)\in \Iln$ as
\begin{equation}
\V^p_{(L,R)}:=\left \{0 \oplus f \in   \Hom(\uL, \uR/\uL) \oplus \left( \Skp \right)=T_{(L,R)}\Iln \,\middle|\,f\in \muL{p} \right \} .
\end{equation}
\end{defn}
It is clear from \ref{eq: tower of filtrations on polys} that 
\begin{equation}
\underbrace{\V^{0}}_{=0}\subset \V^{1}\subset \ldots \subset \V^{k+2}
\end{equation}
and the biggest vertical distribution $\V^{k+2}$ is just the vertical distribution with respect to the projection $\Iln \to \Gr(\uR,l)$. The terminology \emph{vertical} distribution stems from the fact that we will quotient $\Iln$ by these distributions to obtain the manifolds $M^{q}$ in the tower $
\Iln \to M^{k} \to \ldots \to M^0 \to M^{-1}
$
and hence the $\V^{p}$ are indeed vertical distributions.\par
The fact that we are allowed to quotient follows from the next
\begin{lem}\label{lem:verticals_involutive}
All higher vertical distributions $\V^{p}$ are involutive, their integral leaves are affine spaces and their spaces of leaves are manifolds.
\end{lem}
\begin{proof}
The claim is clear for $\V^{k+2}$ since this is the vertical distribution of the projection $\Iln=\Gr(\uR,l)\times \Jkp_{\thetak}  \to \Gr(\uR,l)$. To check the claim for the other vertical distributions note that, since each $\V^p\subset \V^{k+2}$, it suffices to verify it on each fiber of $\Iln \to \Gr(\uR,l)$. But each fiber $\restr{\Iln}{\uL}$ is the affine space $\Jkp_{\thetak}$ and the distribution $\restr{\V^p}{\Jkp_{\thetak}}$ is a flat affine distribution there
\begin{equation}
\restr{\V^p}{\Jkp_{\thetak}}=\muL{p}\times \Jkp_{\thetak} \subset \left(\Skp\right)\times \Jkp_{\thetak}=T\Jkp_{\thetak}.
\end{equation} 
Hence the integral leaves are parallel affine subspaces of $\Jkp_{\thetak}$ modeled on the vector space $\muL{p}$ and the space of leaves is a smooth manifold.
\end{proof}
\begin{defn}
The space of integral leaves of the distribution $\V^p$ is denoted with $M^q$ where $q=k+1-p$ is the complementary degree.
\end{defn}
This way we get the tower of fiber bundles
\begin{equation}\label{eq:tower of fibrations without distr}
M^{k+1} \to M^{k} \to \ldots \to M^0 \to \underbrace{M^{-1}}_{ = \Gr(\uR,l)}
\end{equation}
where each $M^q$ is a bundle over $M^{q-1}$ with affine fibers.\par
\subsection{Identifying $\Il$ and $\pJ$ in the tower}
It is clear that the highest component $M^{k+1}=\Iln$ and that the lowest $M^{-1}=\Gr(\uR,l)$. The second highest $M^{k}$ is $\Il$ by the next
\begin{lem}\label{lem:Vk+1}
The distribution $\V^{1}$ is the vertical distribution of the fibration $\Iln \to \Il$, so $M^{k}=\Il$.
\end{lem}
\begin{proof}
If $(L,R_{\thetakp})$ and $(L,R_{\thetakp'})$ are in the same fiber of $\pr_{l}:\Iln \to \Il$ then $\thetakp-\thetakp'\in\Skp$ is a polynomial vanishing when taking one derivative in direction of $\uL$ since $L\subset \left(R_{\thetakp}\cap R_{\thetakp'} \right)$, which is equivalent by definition to $\thetakp-\thetakp'\in\muL{1}$.
\end{proof}
Next we have
\begin{lem}
$M^0=\pJ$.
\end{lem}
\begin{proof}
Two flags $(L,R_{\thetakp})$ and $(L',R_{\thetakp'})$ are in the same leaf of the distribution $\V^{k+1}$ if and only if $\uL=\underline{L'}$ and $\thetakp-\thetakp'\in \muL{k+1}$, but this last condition is precisely the condition that all $k+1\St$ derivatives of $\thetakp$ and $\thetakp'$ in direction of $\uL$ agree, hence they define the same partial jet prolongation of $\thetak$.
\end{proof}
So we have identified the following components:
\begin{equation}\label{eq:tower_fibrations}
\underbrace{M^{k+1}}_{=\Iln} \to \underbrace{M^{k}}_{=\Il} \to M^{k-1} \to \ldots \to \underbrace{M^{0}}_{=\pJ}  \to  \underbrace{M^{-1}}_{=\Gr(\uR,l)} .
\end{equation}
%
%
%
%
%
%
%%%%%%%%%%%%%%%%%%%%%%%%%%%
%
%
%
\section{Supplying the tower with distributions}\label{sec:supplying with distr}
Our next aim is to supply each $M^q$ with a natural distribution $\uF{q}$. We proceed by exhibiting a second chain of distributions on $\Iln$ which will then descend to the $M^q$'s by a process of symmetry reduction.
\subsection{Higher flag distributions on $\Iln$}
\begin{defn}\label{defn:higher_flag}
For $p=-1,0,1,\ldots,k+1$ define the $p\Th$ \emph{flag distribution} $\F^p$ on $\Iln$ as the sum of  $\V^{p+1}$ with the distribution vertical to the projection $\pr_{n}:\Iln \to \Jkp_{\thetak}$.
\end{defn}
So the plane of the $p\Th$ flag distribution at a point $(L,R)$  is
\begin{equation}
\F^p_{(L,N)}=\left \{  h\oplus f\in \Hom(\uL, \uR/\uL) \oplus \left( \Skp \right)=T_{(L,R)}\Iln \,\middle|\, f\in \muL{p+1} \right \} .
\end{equation}
It is clear that 
\begin{equation}
\underbrace{\F^{-1}}_{=V\pr_{n}}\subset \F^{0}\subset \ldots \subset \underbrace{\F^{k+1}}_{=T\Iln} .
\end{equation}
Concerning the second smallest distribution  $\F^{0}$ we have
\begin{lem}\label{lem:F0 is flag distr}
$\F^{0}$ is the flag distribution $\Dub$ of $\Iln$.
\end{lem}

\begin{proof}
This is a direct consequence of the definitions and lemma \ref{lem:Vk+1}.
\end{proof}
\begin{rem}
We shall see later (corollary \ref{cor:flags are derived of each other}), that the \emph{higher flag} distributions are derived distributions (in the sense defined in subsection \ref{subs:convent jets}) of $\F^{0}$. This, together with the previous lemma \ref{lem:F0 is flag distr}, justifies the terminology.
\end{rem}
To explain how the distributions $\F^{p}$ descend to $M^{q}$ we recall the notion of characteristic symmetries of a distribution \cite{MR1670044}.
\begin{defn}
A vector field $X$ is called a \emph{characteristic symmetry} of a distribution $\mathcal{E}$, if it is contained in $\mathcal{E}$ and a symmetry of $\mathcal{E}$ (so Lie brackets of $X$ with fields in the distribution remain in the distribution). 
\end{defn}
Characteristic symmetries form an involutive sub-distribution of $\mathcal{E}$ and one may always locally quotient $\mathcal{E}$ by the characteristic distribution to obtain a distribution on the space of integral leaves of the characteristic distribution. We call this process the \emph{reduction of $\mathcal{E}$ by characteristic symmetries}.\par
Hence, to proceed, our aim is to prove the following
\begin{thm}\label{thm:charac_sym_of_flag_dist}
For all $p=0,\ldots,k+1$ the characteristic distribution of $\F^p$ is $\V^{p}$.
\end{thm}
To achieve this we construct an explicit basis of the higher flag distribution using local coordinates and compute its commutation relations in the following subsection.
\subsection{Local coordinates, a non-holonomic frame and commutators}
We start by introducing local coordinates on each component of the splitting $\Iln =\Gr(\uR,l) \times \Jkp_{\thetak}$.
Since we will later introduce a second set of local coordinates on $\Iln$, adapted to the projections $\Iln\to M^{q}$, we call this first  set \emph{trivial} and the second \emph{adapted}.
\subsubsection{Trivial local coordinates}\label{subsec:trivial_coordinates}
As in subsection \ref{subsec:infinitesimalpasting} we use affine coordinates $A_{i,j}$ on $\Gr(\uR,l)$ and identify the second component $\Jkp_{\thetak}$ with the vector space $\Skp$ using the chosen \virg{origin jet} \ref{eq:basejet}. \par
Using the bases \ref{eq:y's}, \ref{eq:x's} and \ref{eq:e's} of subsection \ref{subsec:infinitesimalpasting}, a basis of $\Skp$ is given by divided powers
\begin{equation}\label{eq:divided_powers}
\frac{1}{\delta!\lambda!}y^\delta x^\lambda \otimes e_h
\end{equation}
with $|\delta|+|\lambda|=k+1$. Here the factorial $\delta!$ of a multiindex is $\delta_1 ! \cdots \delta_d !$. 
\begin{defn}\label{defn:chart_u}
The dual basis to the divided powers \ref{eq:divided_powers} will be denoted with $u_{\delta,\lambda}^h$ and serves as local coordinates on $\Jkp_{\thetak}$. The coordinates 
\begin{equation}
A_{i,j},u_{\delta,\lambda}^h
\end{equation}
are called \emph{trivial coordinates} on $\Iln$.
\end{defn}
\subsubsection{A non-holonomic frame adapted to vertical distributions} 
Recall that a tangent vector at a point $(\uL,R)\in \Iln$ can be identified with an element $h \oplus f \in \Hom(\uL,\frac{ \uR}{\uL}) \oplus \left( \Skp \right)$. 
Such an $h \oplus f$  is in the distribution $\V^{p}$ at the point $(L,R)$ iff $f \in \muL{p}$ and $h=0$. 
Hence, according to lemma \ref{lem:basis_filt_poly} and the definition of the coordinates $A_{i,j}$, the \virg{partially} divided powers 
\begin{equation}\label{eq:partialdividedpowers}
\left( \frac{1}{\delta!}(y-\sum Ax)^\delta x^\lambda \right) \otimes e_h, \textrm{ with } |\lambda| < p
\end{equation}
form a basis of $\V^{p}$ at each point of $\Iln$ (we have suppressed the component $h=0$). 
In the previous equation the notation $(y- \sum Ax)^\delta$ stands for $(y_1- \sum_jA_{1,j}x_j)^{\delta_1}\cdots (y_d- \sum_jA_{d,j}x_j)^{\delta_d}$. 
\begin{defn}\label{defn:vertical_fields}
Local vector fields on $\Iln$ corresponding to the partially divided powers \ref{eq:partialdividedpowers}  will be denoted with $V_{\delta,\lambda}^h$ and called \emph{vertical fields}.
\end{defn}
These vertical fields $V_{\delta,\lambda}^h$ together with the coordinate fields $\partial_{A_{i,j}}$ clearly form a (non-holonimic) local frame on $\Iln$.\par
The $V_{\delta,\lambda}^h$ will play an analogous role to the vertical coordinate fields $\partial_{u^j_\sigma}$ on jet spaces \cite{MR1670044} while the $\partial_{A_{i,j}}$ will play an analogous role to the total derivatives $D_i$ on jet spaces. For this reason and since we later introduce a second set of coordinates in which the current $\partial_{A_{i,j}}$ will have a different expansion, we adopt the following terminology.
\begin{defn} The fields $\partial_{A_{i,j}}$ from the current chart will be denoted with $D_{i,j}$ and called \emph{homogeneous total derivatives}.
\end{defn}
\begin{rem}
The adjective \emph{homogeneous} will be justified after comparing the commutation relations \ref{eq:commutators} and the expansion \ref{eq:Dexpansion} of the $D_{i,j}$, with the analogous commutation relations and expansion of classical total derivatives $D_{i}$ on jet spaces.
\end{rem}
It is evident that the frame $D_{i,j},\,  V_{\delta,\lambda}^h$ is adapted to the higher vertical- and flag distributions in the sense that
\begin{align}
\V^{p} &=\left \langle V_{\delta,\lambda}^h \, \middle | \,  |\lambda| < p \right \rangle \\
\F^p &=\left \langle D_{i,j},\,  V_{\delta,\lambda}^h \, \middle | \,  |\lambda| \leq p \right \rangle
\end{align}
To prove that the distributions $\V^{p}$ are the characteristics of $\F^{p}$ we compute the commutators of the frame.
\begin{thm}\label{thmCommutatorsFrame}
All commutators of the frame $D_{i,j}, V_{\delta,\lambda}^h$ are zero except for the commutators $\left[V_{\delta,\lambda}^h, D_{i,j}\right]$ when $\delta_i>0$. In that case we have:
\begin{equation}\label{eq:commutators}
\left[V_{\delta,\lambda}^h,D_{i,j}\right]=V_{\delta-1_i,\lambda+1_j}^h .
\end{equation}
\end{thm}

\begin{proof}
That $[D_{i,j},D_{i',j'}]=0$ is clear since in the chosen coordinates these fields are just partial derivatives. 
That $[V_{\delta,\lambda}^h,V_{\delta',\lambda'}^{h'}]=0$ is also easily seen, since by equation \ref{eq:partialdividedpowers}, the $V_{\delta,\lambda}^h$ are linear combinations of the coordinate fields $\partial_{u^h_{\delta,\lambda}}$ with coefficients depending only on the coordinates $A_{i,j}$.

We are left to consider the Lie brackets $\left[V_{\delta,\lambda}^h,D_{i,j}\right]$. We compute how these act on coordinate functions. First note that $\left[V_{\delta,\lambda}^h,D_{i,j}\right](A_{i',j'})=0$ since 

\begin{equation}
\left[V_{\delta,\lambda}^h,D_{i,j}\right](A_{i',j'})=V_{\delta,\lambda}^h( \underbrace{D_{i,j}(A_{i',j'})}_{=\textrm{constant}})-D_{i,j} (\underbrace{V_{\delta,\lambda}^h (A_{i',j'})}_{=0})=0 .
\end{equation}
Now consider the action of $\left[V_{\delta,\lambda}^h, D_{i,j}\right]$ on a coordinate function $u_{\Delta,\Lambda}^H$ where $\Delta\in \mathbb{N}^d$ and $\Lambda \in \mathbb{N}^n$ are multi-indices and $H=1,\ldots, m$:

\begin{align}\label{eq:liebracket_on_u}
V_{\delta,\lambda}^h( \underbrace{D_{i,j}(u_{\Delta,\Lambda}^H)}_{=0})- D_{i,j} (V_{\delta,\lambda}^h (u_{\Delta,\Lambda}^H))  & =-D_{i,j} (V_{\delta,\lambda}^h (u_{\Delta,\Lambda}^H)) .
\end{align}
To continue the computation consider the inner term $V_{\delta,\lambda}^h (u_{\Delta,\Lambda}^H)$ on the r.h.s. When $h\neq H$ this is obviously $0$. In the case $h = H$ note that $V_{\delta,\lambda}^h (u_{\Delta,\Lambda}^h)$ is the coefficient in front of $\partial_{u_{\Delta,\Lambda}^h}$ in the expansion of $V_{\delta,\lambda}^h$ in the coordinate frame. But this is the same as the coefficient in the expansion of $\frac{1}{\delta!}(y-\sum Ax)^\delta x^\lambda$ in front of  $\frac{1}{\Delta!\Lambda!}y^\Delta x^\Lambda$. 
This coefficient may be computed by applying the operator  
\begin{equation}
\partial_y^\Delta \partial_x^\Lambda:=\partial_{y_1}^{\Delta_1}\cdots \partial_{y_d}^{\Delta_d}\partial_{x_1}^{\Lambda_1}\cdots \partial_{x_l}^{\Lambda_l}
\end{equation} 
to $\frac{1}{\delta!}(y+\sum Ax)^\delta x^\lambda$ since all polynomials involved are homogenous. So we have
\begin{equation}\label{Vu_coefficient}
V_{\delta,\lambda}^h (u_{\Delta,\Lambda}^h) = \partial_y^\Delta \partial_x^\Lambda \left ( \frac{1}{\delta!} (y-\sum Ax)^\delta x^\lambda \otimes e_h \right ) .
\end{equation}
Plugging this in in the r.h.s of equation \ref{eq:liebracket_on_u} we obtain
\begin{equation}
\left[V_{\delta,\lambda}^h, D_{i,j} \right](u_{\Delta,\Lambda}^h) = -\partial_{A_{i,j}} \partial_y^\Delta \partial_x^\Lambda \left ( \frac{1}{\delta!} (y-\sum Ax)^\delta x^\lambda \otimes e_h \right ) .
\end{equation}
Now we can exchange the order of derivatives on the r.h.s and derive first w.r.t. $\partial_{A_{i,j}}$. Using the chain rule we compute:
\begin{align}
\partial_{A_{i,j}} \left( \frac{1}{\delta!} (y-\sum Ax)^\delta x^\lambda \right)  & = 
-\delta_i \cdot x_j \frac{1}{\delta!} (y-\sum Ax)^{\delta-1_i} x^\lambda \\
& = \begin{cases} 0 &\mbox{if } \delta_i = 0 \\
-\frac{1}{(\delta-1_i)!} (y+\sum Ax)^{\delta-1_i} x^{\lambda+1_j} & \mbox{if } \delta_i > 0. \end{cases} 
\end{align}
So we arrive at:
\begin{equation}
\left[V_{\delta,\lambda}^h,D_{i,j}\right](u_{\Delta,\Lambda}^h) = \begin{cases} 0 &\mbox{if } \delta_i = 0 \\
 \partial_y^\Delta \partial_x^\Lambda \left ( \frac{1}{(\delta-1_i)!} (y+\sum Ax)^{\delta-1_i} x^{\lambda+1_j} \otimes e_h \right ) & \mbox{if } \delta_i > 0. \end{cases} 
\end{equation}
From this we conclude that $\left[V_{\delta,\lambda}^h, D_{i,j} \right]=0$ if $\delta_i=0$ while in the case when $\delta_i>0$ the r.h.s. of the last equation is precisely $V_{\delta-1_i,\lambda+1_j}^h(u_{\Delta,\Lambda}^h)$ by equation \ref{Vu_coefficient}.
\end{proof}
A remarkable direct consequence of the commutation relations \ref{eq:commutators}, which we shall not need in the remainder, is
\begin{cor}\label{cor:flags are derived of each other}
All flag distributions $\F^{p}$ with $p\geq 1$ are derived distributions of the flag distribution $\Dub=\F^{0}$. More precisely
\begin{equation}
\F^{p+1}=\left[\F^{p}\F^{p}\right]
\end{equation}
for all $p=0, \ldots ,k$.
\end{cor}
\subsection{The reduced distributions and identifying the polar distribution}
Commutation relations \ref{eq:commutators} immediately imply theorem \ref{thm:charac_sym_of_flag_dist}, hence the flag distribution $\F^{p}$ reduces to a distribution on $M^{q}$ for $q=k+1-p$ and $0\leq q \leq k+1$ by quotienting out characteristic symmetries. 
\begin{defn}
The \emph{reduction of the flag distribution $\F^{p}$ to $M^{q}$} is denoted with $\uF{q}$, where $q=k+1-p$ is the complementary degree to $p$.
\end{defn}
Since $\F^{k+1}=T\Iln$ we have $\uF{0}=TM^{0}$. Further since $\V^{0}=0$ and  $\F^{0}=\Dub$ by lemma \ref{lem:F0 is flag distr}, we have $\uF{k+1}=\Dub$. So the tower \ref{eq:tower of fibrations without distr} is now enhanced with distributions as follows:
\begin{equation}\label{eq:tower_distr}
\underbrace{(M^{k+1}, \uF{k+1})}_{=(\Iln,\Dub)} \to (M^{k}, \uF{k}) \to (M^{k-1},\uF{k-1}) \to \ldots \to (M^0,\underbrace{\uF{0}}_{=TM^{0}} )  \to  \Gr(\uR,l)
\end{equation}
%\begin{equation}\label{eq:tower_distr}
%\underbrace{(M^{k+1}, \uF{k+1})}_{=(\Iln,\Dub)} \to \underbrace{(M^{k}, \uF{k})}_{=(\Il,\P)} \to (M^{k-1},\uF{k-1}) \to \ldots \to \underbrace{(M^0,\uF{0})}_{=(\pJ,T\pJ)}  \to  \Gr(\uR,l)
%\end{equation}
%
%
We already established $M^{k}=\Il$. We now claim that $\uF{k}$ is the polar distribution $\P$. For this it suffices to prove the following
\begin{prop}\label{prop:second_flag_is_preim_polar}
$\F^{1}$ is the lift of the polar distribution $\P$ from $\Il$ to $\Iln$ via $\Iln \to \Il$.
\end{prop}
\begin{proof}
Fix $(L,R)\in \Iln$. By the definition of $\F^{1}$ we need to show that for any tangent vector at $(L,R)\in \Iln$ of the form $0\oplus f \in \Hom( \uL, \uR/\uL ) \oplus \left( \Skp \right)$ the following conditions are equivalent:
\begin{itemize}
\item[1)] $f\in \muL{2}$
\item[2)] $T_{(L,R)}\pr_{l}(0\oplus f)\in \P_{L}$
\end{itemize}
where $T_{(L,R)}\pr_{l}$ is the tangent map of $\pr_{l}:\Iln\to \Il$ at $(L,R)$. Let 
\begin{equation}
\de f : \uR \to \Sk
\end{equation} 
denote the total differential of the polynomial $f \in \Skp$ and let
\begin{equation}\label{eq:dprest}
\restr{\de f}{L}:L\to \Ctheta/L
\end{equation}
denote its restriction to $L$. In \ref{eq:dprest} we have implicitly used the canonical isomorphism $L\cong \uL$ and the natural inclusion $\Sk \subset \Ctheta/L$ as vertical tangent space to the projection $\Jk\to\Jkl$. It is straightforward to see that for any $0\oplus f \in T_{(L,R)}\Iln$
\begin{equation}
T_{(L,R)}\pr_{l}(0\oplus f)= \restr{\textrm{d} f}{L} .
\end{equation}
Now we  compute with $l_1,l_2\in L$
\begin{align}
\Omega\left (l_{1}\, , \,T_{(L,R)}\pr_{l}(0\oplus f)(l_{2})\right)&=\Omega\left (l_{1}\, , \,\restr{\de f}{L}\left (l_{2} \right)\right)\\
&=\partial_{l_{1}}\partial_{l_{2}}f \label{eq:omegastruct}
\end{align}
where \ref{eq:omegastruct} follows from the structural properties of the metasymplectic form $\Omega$ \ref{metasymplecticnontric}. But \ref{eq:omegastruct} is zero for all $l_1,l_2\in L\cong \uL$ iff $f\in \muL{2}$ so the claim follows from description \ref{eq:descrpolar}.
\end{proof}
It remains to identify the pasting conditions in the tower. We will do this in the next section together with the proof that consecutive components of the tower are prolongations.
\section{Proving that the tower prolongs the pasting conditions}\label{sect:proof that tower is prolongations}
\subsection{Consecutive $M^{q}$'s are prolongations}
Our next aim is to prove that each distributions $(M^{q},\uF{q})$ is the prolongation of the previous $(M^{q-1},\uF{q-1})$ for $q>1$. Denote the projection with 
\begin{equation}
\Pi_{q,q-1}:M^{q}\to M^{q-1}
\end{equation}
and let $\phiq$ be a point in the fiber $M^{q}_{\phiql}$ over $\phiql \in M^{q-1}$. Attached to $\phiq$ is the plane $\uF{q}_{\phiq}$ of the distribution $\uF{q}$ which we may project down to $M^{q-1}$. We denote the projected plane with 
\begin{equation}
\Qphiq:=T_{\phiq}\Pi_{q,q-1}(\uF{q}_{\phiq}) .
\end{equation}
These \virg{Q-planes} are analogous to the R-planes in jets spaces by the following three results which together prove that each $(M^{q},\uF{q})$ is the prolongation of $(M^{q-1},\uF{q-1})$ for $q>1$.
\begin{prop}\label{prop:prolongation1}
For each $\phiq\in M^{q}$ with $q=1,\ldots,k+1 $, the plane $\Qphiq$ is a horizontal maximal integral element in $(M^{q-1},\uF{q-1})$ of dimension $\dim \Gr(\uR,l)$. Horizontal here means transversal to $M^{q-1}\to M^{q-2}$, which turns out to be equivalent to being transversal to $M^{q-1}\to \Gr(\uR,l)$.
\end{prop}
\begin{prop} \label{prop:prolongation2}
For all $q=1,\ldots,k+1$, the map 
\begin{equation}
\phiq \mapsto \Qphiq
\end{equation}
is an injection from the fiber $M^{q}_{\phiql}$ into the space of horizontal maximal integral elements of $\uF{q-1}$ at $\phiql$. 
\end{prop}
So we may identify $M^{q}$ with a subset of maximal horizontal integral elements of $(M^{q-1},\uF{q-1})$. In fact, for $q\geq 2$, \emph{any} maximal integral elements of $(M^{q-1},\uF{q-1})$ is of the form $\Qphiq$, which is the content of the next
\begin{prop} \label{prop:prolongation3}
For all $q=2,\ldots,k+1$, the map 
\begin{equation}
\phiq \mapsto \Qphiq
\end{equation}
is a surjection from the fiber $M^{q}_{\phiql}$ to horizontal maximal integral elements of $\uF{q-1}$ at $\phiql$.
\end{prop}

To prove propositions \ref{prop:prolongation1}, \ref{prop:prolongation2}, \ref{prop:prolongation3} we introduce a second set of coordinates on $\Iln$ which descend to the quotients $M^q$. This allows us to give explicit bases of the reduced distributions $\uF{q}$ and compute their commutation relations.\par
\subsection{Local coordinates and non holonomic frames on the $M^{q}$'s}
Since we fixed a jet $\theta_{k+1,\text{orig}} \in \Jkp_{\thetak}$ in \ref{eq:basejet} to identify $\Jkp_{\thetak}$ with the vector space $\Skp$, we may consider
\begin{equation}
\Gr(\uR,l)\times \Jkp_{\thetak} \to \Gr(\uR,l)
\end{equation} 
to be a vector bundle. The partially divided powers $\frac{1}{\delta!}(y-\sum Ax)^\delta x^\lambda\otimes e^h$ then form a basis in each fiber. This frame is \virg{moving} from fiber to fiber as it depends on the base coordinates $A_{i,j}$. Here $A_{i,j}$ and $x,y$ have the same meaning as in subsection \ref{subsec:trivial_coordinates}. \par
\begin{defn}\label{defn:chart_v}
The fiber-wise dual one-forms to the frame 
\begin{equation}
\frac{1}{\delta!}(y-\sum Ax)^\delta x^\lambda\otimes e_h
\end{equation} 
will be denoted with $v_{\delta,\lambda}^h$ and provide new coordinates on the fibers of $\Gr(\uR,l)\times \Jkp_{\thetak} \to \Gr(\uR,l)$. Together with the coordinates $A_{i,j}$ on the base $\Gr(\uR,l)$ they constitute another set of local coordinates on  $\Iln$ which we call \emph{adapted}. 
\end{defn}
Observe that in these adapted coordinates the vector fields $V_{\delta,\lambda}^h$ are just the partial derivatives $\partial_{v_{\delta,\lambda}^h}$ 
\begin{equation}\label{eq:V's_in_adaptedcoord}
V_{\delta,\lambda}^h=\partial_{v_{\delta,\lambda}^h}
\end{equation}
while the fields $D_{i,j}$ are \emph{no longer} the coordinate fields $\partial_{A_{i,j}}$, as in the trivial coordinates.\par
It is clear from \ref{eq:V's_in_adaptedcoord} that the coordinates $A_{i,j},v^{h}_{\delta,\lambda}$ with $|\delta| \leq q$ descend to coordinates on $M^{q}$.\par
Our next aim is to expand the fields $D_{i,j}$ in the coordinates $A_{i,j},v_{\delta,\lambda}^h$. 
\begin{prop}\label{prop:D expansion in coord}
We have
\begin{align}
	D_{i,j}(A_{i',j'}) &= 
		\begin{cases}
			1  &\mbox{if } i = i' \textrm{ and } j=j' \\
			0 & \mbox{else } 
		\end{cases} \label{D_ijOnA_ij}\\
	D_{i,j}(v_{\delta,\lambda}^h) &= 
		\begin{cases}
			v_{\delta+ 1_i, \lambda - 1_j}^h  &\mbox{if } \lambda_j>0 \\
			0 & \mbox{else } 
		\end{cases} \label{D_ijOnv's} ,
\end{align}
from which the coordinate expansion 
\begin{equation}\label{eq:Dexpansion}
D_{i,j}=\partial_{A_{i,j}}+\sum_{ \lambda_j>0}v_{\delta+ 1_i, \lambda - 1_j}^h\partial_{v_{\delta,\lambda}^h}
\end{equation}
follows. The sum on the r.h.s. of \ref{eq:Dexpansion} runs over all repeated indices $h,\delta,\lambda$.
\end{prop}

\begin{proof}

Equation \ref{D_ijOnA_ij} is obvious if we recall that in the previous trivial coordinates the derivations $D_{i,j}$ were just the partial derivative with respect to $A_{i,j}$. 

To prove the second equation \ref{D_ijOnv's} we first express the $u^h_{\delta, \lambda}$ and $v^h_{\delta, \lambda}$ as sections of the dual $S^{k+1}\underline{R}\otimes N^*$ using the dual basis to $y_1, \ldots , y_d,x_1, \ldots , x_l\in\underline{R}^*$ and $e^{*}_1,\ldots ,e^{*}_m\in N^*$ and the natural isomorphism
\begin{equation}
S^{k+1}(\uR^{*})\cong(S^{k+1}\uR)^{*}\label{eq:identification of symmetric product of duals}
\end{equation}
induced from the non-degenerate pairing
\begin{equation}
S^{k+1}\uR \otimes S^{k+1}(\uR^{*}) \rightarrow \mathbb{R}
\end{equation}
 given by
\begin{equation}
w_{1}\cdot\ldots\cdot w_{k+1}\otimes\alpha_{1}\cdot\ldots\cdot\alpha_{k+1}  \mapsto  \sum_{\varsigma}\prod_{i=1}^{k+1}\langle w_{\varsigma(i)},\alpha_{i}\rangle
\end{equation}
where $\varsigma$ runs through all permutations of the set $\{1,\ldots,k+1\}$. If $r_{1},\ldots,r_{n}$ is a basis
of $\uR$ and the associated dual basis of $\uR^{*}$ is denoted with
$r_{1}^{*},\ldots,r_{n}^{*}$, then under identification \ref{eq:identification of symmetric product of duals}
the dual basis of $r^{\sigma}\in S^{k+1}\uR$ is mapped to $\frac{1}{\sigma!}(r^{*})^{\sigma}\in S^{k+1}\uR^{*}$.\par
So letting $y^*_1$, \ldots , $y^*_d$, $x^*_1$, \ldots , $x^*_l\in \underline{R}$ denote the basis dual to $y_1, \ldots , y_d,x_1, \ldots , x_l\in\underline{R}^*$ and $e^{*}_1,\ldots ,e^{*}_m\in N^*$ the one dual to $e_1,\ldots ,e_m\in N$, we have
\begin{equation}
u_{\delta,\lambda}^h=\left ( y^* \right )^\delta \left ( x^* \right )^\lambda \otimes e^{*}_h.
\end{equation}
Further, since the basis of $\underline{R}$ dual to the basis 
\begin{equation}
\left( y_1-\sum A_{1,j}x_j\right),\, \ldots ,\,  \left ( y_d-\sum A_{d,j}x_j\right ),\,  x_1,\, \ldots ,\, x_l
\end{equation}
of $R^*$ is given by 
\begin{equation}
y^*_1,\, \ldots, \, y^*_d,\,  \left( x^*_1+\sum A_{i,1}y^*_i \right), \, \ldots, \, \left ( x^*_l+\sum A_{i,l}y^*_i \right )
\end{equation}
we have
\begin{equation}\label{dualbasis}
v_{\delta,\lambda}^h=\frac{1}{\lambda!} \left ( y^* \right )^\delta \left ( x^*+   A y^* \right )^\lambda \otimes e^{*h}
\end{equation}
again by \ref{eq:identification of symmetric product of duals} and since the $v_{\delta,\lambda}^h$ are by definition dual to the basis $\lambda!\frac{1}{\delta!\lambda!}(y-\sum Ax)^\delta x^\lambda$. By expanding the powers on the r.h.s. of \ref{dualbasis} we could express the coordinates $v_{\delta,\lambda}^h$ as linear combinations of the $u_{\delta,\lambda}$ with coefficients depending on the variables $A_{i,j}$. We shall not do this, instead we recall again that in the coordinates $u_{\delta,\lambda},A_{i,j}$ the derivations $D_{i,j}$ act as partial derivative with respect to $A_{i,j}$. Hence applying the chain rule we can compute

\begin{align}
D_{i,j}(v_{\delta,\lambda}^h) &=\frac{\partial}{\partial {A_{i,j}}}\left ( \frac{1}{\lambda!} \left ( y^* \right )^\delta \left ( x^*+  A y^* \right )^\lambda  \otimes e^*_h \right ) \\
&= \lambda_j y^*_i \frac{1}{\lambda!} \left ( y^* \right )^\delta \left ( x^*+   A y^* \right )^{\lambda-1_j}\otimes e^*_h \\
&= \begin{cases}
			\frac{1}{(\lambda-1_j)!} \left ( y^* \right )^{\delta+1_i } \left ( x^*+   A y^* \right )^{\lambda-1_j} \otimes e^*_h &\mbox{if } \lambda_j>0 \\
			0 & \mbox{else } 
		\end{cases}\\
&= \begin{cases}
			v_{\delta+ 1_i, \lambda - 1_j}^h  &\mbox{if } \lambda_j>0 \\
			0 & \mbox{else } 
		\end{cases} .
\end{align}
\end{proof}

\begin{defn}\label{defn:truncuatedD}
The \emph{$q$-truncated homogeneous total derivatives} are the vector fields on $M^{q}$ (where $q\geq 0$) defined in local adapted coordinates by
\begin{equation}\label{eq:truncuatedD}
D_{i,j}^{[q]}:=\partial_{A_{i,j}}+\sum_{\substack{  |\delta|<q \\\lambda_j>0 }}v_{\delta+ 1_i, \lambda - 1_j}^h\partial_{v_{\delta,\lambda}^h} .
\end{equation}
\end{defn}
It is clear that $D_{i,j}^{[k+1]}=D_{i,j}$.
\begin{lem}\label{lem:commutatiors of reduced distr}
\begin{itemize}
\item[a)] The fields $\left \{ \partial_{v_{\delta,\lambda}^h}, D_{i,j}^{[q]}\right \}$ with $|\delta| \leq q$ make up a frame on $M^{q}$.
\item[b)] Commutators of this frame are all zero except for the commutators 
\begin{equation}\label{eq:commutatorsrtrunc}
\left[\partial_{v_{\delta,\lambda}^h},D_{i,j}^{[q]}\right]=\partial_{v_{\delta-1_i,\lambda+1_j}^h}
\end{equation}
when $|\delta_{i}|>0$.
\item[c)] The vertical distribution of $M^{q}\to M^{q-1}$ is spanned by $\partial_{v_{\delta,\lambda}^h}$ with $|\delta|=q$.
\item[d)] The fields $\left\{ \partial_{v_{\delta,\lambda}^h}, D_{i,j}^{[q]}\right \}$ with $|\delta|=q$ form a local basis of $\uF{q}$ and split it into vertical and \emph{horizontal} part. 
\end{itemize}
\end{lem}

\begin{proof}
Straightforward from the definitions and the previous results.
\end{proof}
\begin{cor}\label{lem:splitting uF}
For $q=0,\ldots , k+1$ any plane $Q\subset \uF{q}_{\phiq}$ of maximal dimension and horizontal to $M^{q}\to M^{q-1}$ has a basis of the form
\begin{equation}\label{eq:basis of generic Q}
C_{i,j}:=D_{i,j}^{[q]}+\sum_{|\delta|=q}C^{\delta, \lambda}_{i,j,h}\partial_{v_{\delta,\lambda}^h}
\end{equation}
with unique coefficients $C^{\delta, \lambda}_{i,j,h}$. It is hence of dimension $\dim\Gr(\uR,l)$ and horizontal to the projection $M^{q}\to \Gr(\uR,l)$.
\end{cor}
\begin{defn}
We denote the curvature form of $\uF{q}$ with $\Omega^{[q]}$. We may compute with it directly by using commutators \ref{eq:commutatorsrtrunc}.
\end{defn} 
\begin{lem}\label{lem:coefficents of integral Q's}
For $q=1,\ldots , k+1$ a horizontal plane $Q\subset \uF{q}_{\phiq}$ of dimension $\Gr(\uR,l)$ is an integral element of $\uF{q}$ if and only if the coefficients $C^{\delta, \lambda}_{i,j,h}$ of its basis \ref{eq:basis of generic Q} satisfy
\begin{equation}\label{conditionC1}
C^{\delta,\lambda}_{i,j,h} = C^{\delta',\lambda'}_{i',j',h} %\quad \text{whenever} \quad \delta+1_{i}=\delta'+1_{i'} \quad \text{and} \quad \lambda-1_{j}=\lambda'-1_{j'}
\end{equation}
whenever the indices satisfy
\begin{align}
\delta_{i'}>0,& \; \delta'_{i}>0 \label{conditionC1a}\\ 
\delta-1_{i'}&=\delta'-1_{i} \label{conditionC1b} \\
 \lambda+1_{j'}&=\lambda'+1_{j} \label{conditionC1c}
\end{align}
and condition
\begin{equation}
C^{\delta,\lambda}_{i,j,h} = 0 \quad \text{whenever} \quad \lambda_{j}=0 \quad \text{and} \quad l>1. \label{conditionC2}
\end{equation}
\end{lem}
\begin{proof}
The plane $Q$ is integral if and only if
\begin{equation}\label{eq:conditions on basis}
\Omega^{[q]}(C_{i,j},C_{i',j'})=0
\end{equation}
for all $i,j,i',j'$.
Expanding the left hand side of \ref{eq:conditions on basis} leads to
\begin{equation}\label{involutivityconditions1}
\sum_{
	\substack{
		|\delta|=q\\
		\delta_{i'}>0
		}
	}C^{\delta,\lambda}_{i,j,h}\partial_{v^{h}_{\delta-1_{i'},\lambda+1_{j'}}}
	-\sum_{
		\substack{
			|\delta|=q\\
			\delta_{i}>0
		}
	}C^{\delta,\lambda}_{i',j',h}\partial_{v^{h}_{\delta-1_{i},\lambda+1_{j}}}=0 .
\end{equation}
Changing indices in the first sum to $\Delta = \delta - 1_{i'}$, $\Lambda = \lambda + 1_{j'}$ and in the second to $\Delta = \delta - 1_{i}$, $\Lambda = \lambda + 1_{j}$ transforms equation \ref{involutivityconditions1}  to
\begin{equation}
\sum_{
	\substack{
		|\Delta|=q-1\\
		\Lambda_{j'}>0
		}
	}C^{\Delta+1_{i'},\Lambda-1_{j'}}_{i,j,h}\partial_{v^{h}_{\Delta,\Lambda}}
	-\sum_{
		\substack{
			|\Delta|=q-1\\
			\Lambda_{j}>0
		}
	}C^{\Delta+1_{i},\Lambda-1_{j}}_{i',j',h}\partial_{v^{h}_{\Delta,\Lambda}}=0	 .
\end{equation}
Collecting bases we find
\begin{multline}\label{eq: comp involutivity conditions}
\sum_{
	\substack{
		|\Delta|=q-1\\
		\Lambda_{j'}>0\\
		\Lambda_{j}>0
		}
	} \left( C^{\Delta+1_{i'},\Lambda-1_{j'}}_{i,j,h} - C^{\Delta+1_{i},\Lambda-1_{j}}_{i',j',h} \right) \partial_{v^{h}_{\Delta,\Lambda}} +\\
+\sum_{
	\substack{
		|\Delta|=q-1\\
		\Lambda_{j'}>0\\
		\Lambda_{j}=0
		}
	}C^{\Delta+1_{i'},\Lambda-1_{j'}}_{i,j,h}\partial_{v^{h}_{\Delta,\Lambda}}
+ \sum_{
		\substack{
			|\Delta|=q-1\\
			\Lambda_{j}>0\\
			\Lambda_{j'}=0
		}
	} C^{\Delta+1_{i},\Lambda-1_{j}}_{i',j',h}\partial_{v^{h}_{\Delta,\Lambda}}=0	.
\end{multline}
Equating coefficients to zero and returning to the previous indices we find conditions  \ref{conditionC1} from the first summand of \ref{eq: comp involutivity conditions}, while from the second and third summands (which are only present when $l>1$) we find condition \ref{conditionC2}.
\end{proof}
\begin{lem}\label{lem:conditions on coefficients of Qphi}
For $q=1,\ldots , k+1$ a horizontal plane $Q\subset \uF{q-1}_{\phiql}$ is of the form $Q_{\phiq}$ for some $\phiq \in M^{q}_{\phiql}$ if and only if the coefficients $C^{\delta, \lambda}_{i,j,h}$ of its basis \ref{eq:basis of generic Q} satisfy
\begin{equation}\label{conditionC1'}
C^{\delta,\lambda}_{i,j,h} = C^{\delta',\lambda'}_{i',j',h} %\quad \text{whenever} \quad \delta+1_{i}=\delta'+1_{i'} \quad \text{and} \quad \lambda-1_{j}=\lambda'-1_{j'}
\end{equation}
whenever the indices satisfy
\begin{align}
\lambda_{j}>0,& \; \lambda'_{j'}>0 \label{conditionC1a'}\\ 
\delta+1_{i}&=\delta'+1_{i'} \label{conditionC1b'} \\
 \lambda-1_{j}&=\lambda'-1_{j'} \label{conditionC1c'}
\end{align}
and condition
\begin{equation}
C^{\delta,\lambda}_{i,j,h} = 0 \quad \text{whenever} \quad \lambda_{j}=0. \label{conditionC2'}
\end{equation}
\end{lem}
\begin{proof}
We start by showing that the basis of a plane $Q_{\phiq}$ satisfies \ref{conditionC1'} and \ref{conditionC2'}. By lemma \ref{lem:commutatiors of reduced distr} the plane $\uF{q}_{\phiq}$ is spanned by the fields $D_{i,j}^{[q]}$ and vertical fields $\partial_{v_{\delta,\lambda}^h}$ with $|\delta|=q$. The vertical ones are annihilated when projecting to $M^{q-1}$ while the $D_{i,j}^{[q]}$ are mapped to 
\begin{equation}\label{eq:basisofQplanes}
C_{i,j}:=D_{i,j}^{[q-1]}+\sum_{\substack{|\delta|=q-1 \\ \lambda_j >0}}v_{\delta+ 1_i, \lambda - 1_j}^h\partial_{v_{\delta,\lambda}^h}
\end{equation} 
where now the numbers $v_{\delta+ 1_i, \lambda - 1_j}^h$ on the r.h.s of \ref{eq:basisofQplanes} are to be understood as the coordinates of the point $\phiq$ in the fiber over $\phiql$. Vectors \ref{eq:basisofQplanes} are a basis of $Q_{\phiq}$ of the form \ref{eq:basis of generic Q} with $C^{\delta,\lambda}_{i,j,h}=v_{\delta+ 1_i, \lambda - 1_j}^h$. It is straightforward to see that these coefficients satisfy \ref{conditionC1'} and \ref{conditionC2'}.\par
Conversely suppose the basis $C_{i,j}$ of a plane $Q\subset \uF{q-1}_{\phiql}$ satisfies conditions \ref{conditionC1'} and \ref{conditionC2'}. We need to find a point $\phiq \in M^{q}_{\phiql}$ such that $Q=Q_{\phiq}$. 
For any multiindex $(\Delta,\Lambda)\in \mathbb{N}^d \times \mathbb{N}^l$ with $|\Delta|=q$, $|\Delta|+|\Lambda|=k+1$ and any $h\in {1,\ldots, m}$ define the numbers
\begin{equation}\label{coordinatesofphiq}
v^{h}_{\Delta,\Lambda}:=C^{\Delta-1_{i},\Lambda+1_{j}}_{i,j,h}
\end{equation}
where we choose $i$ in such a way that $\Delta_{i}> 0$, which is always possible since $|\Delta| \geq 1$. By \ref{conditionC1'}  this definition is independent of the choices of $i,j$. By further taking into consideration condition   \ref{conditionC2'} we see that
\begin{equation}
C_{i,j}=D_{i,j}^{[q-1]}+\sum_{\substack{|\delta|=q-1\\ \lambda_j >0}}v_{\delta+ 1_i, \lambda - 1_j}^h\partial_{v_{\delta,\lambda}^h}
\end{equation}
which by \ref{eq:basisofQplanes} proves that $Q$ is of the form $Q_{\phiq}$ with the point $\phiq \in M^{q}_{\phiql}$ determined by the fiber coordinates \ref{coordinatesofphiq}.

\end{proof}
We are now in the position to easily prove propositions \ref{prop:prolongation1},\ref{prop:prolongation2} and \ref{prop:prolongation3}.
\begin{proof}[Proof of proposition \ref{prop:prolongation1}]
By lemma \ref{lem:conditions on coefficients of Qphi} the basis $C_{i,j}$ of $Q_{\phiq}$ satisfies conditions \ref{conditionC1'} and \ref{conditionC2'}, which for $q>1$ are the same as conditions \ref{conditionC1} and \ref{conditionC2} of lemma \ref{lem:coefficents of integral Q's}, hence $Q_{\phiq}$ is integral. When $q=1$, $Q_{\phi_{1}}$ is integral since $\uF{0}=TM^{0}$.
\end{proof}

\begin{proof}[Proof of proposition \ref{prop:prolongation2}]
If $\phiq \neq \tilde{\phiq}$ are two distinct points over $\phiql$ there must be indices $\delta,\lambda,h$ such that the corresponding fiber coordinates of the points differ $v_{\delta,\lambda}^{h}\neq \tilde{v}_{\delta,\lambda}^{h}$. Since $q>0$ there is an $i$ such that $\delta_{i}\neq 0$. Then the coefficients in front of $\partial_{v_{\delta-1_{i},\lambda+i_{j}}^{h}}$ in the bases \ref{eq:basisofQplanes} of $Q_{\phiq}$ and $Q_{\tilde{\phiq}}$ differ, hence $Q_{\phiq}\neq Q_{\tilde{\phiq}}$ by uniqueness of the bases $C_{i,j}$.
\end{proof}

\begin{proof}[Proof of proposition \ref{prop:prolongation3}]
For the range of indices $q$ under consideration conditions \ref{conditionC1} and \ref{conditionC2} of lemma \ref{lem:coefficents of integral Q's} coincide with conditions \ref{conditionC1'} and \ref{conditionC2'} of lemma \ref{lem:conditions on coefficients of Qphi} hence an integral $Q$ is of the form $Q_{\phiq}$.
\end{proof}
\subsection{Identifying the pasting conditions in the tower}
Finally, the PDEs we called infinitesimal pasting conditions \ref{infpaste1} and \ref{infpaste2} are encoded in $(M^{1},\uF{1})$ as follows.
\begin{prop} \label{prop:prolongation4}
The image of the map 
\begin{equation}
\phi_{1} \mapsto Q_{\phi_{1}},
\end{equation}
understood in the obvious way as a subset of the first order jet space of the bundle $\pJ \to \Gr(\uR,l)$, is precisely the zero set of the infinitesimal pasting conditions \ref{infpaste1} and \ref{infpaste2}.
\end{prop}
\begin{proof}
Observe first that coordinates $A_{i,j}, v_{\lambda}^{h}$ used in the description of the infinitesimal pasting conditions \ref{infpaste1} and \ref{infpaste1} are precisely the adapted coordinates $A_{i,j}, v_{0,\lambda}^{h}$ on $M^{0}$ (where now $\delta=0$). Fix a point $\phi_{0}\in M^{0}$. Any $\dim(\Gr(\uR,l))$-dimensional horizontal plane $Q\subset T_{\phi_{0}}M^{0}$ is now of the form 
\begin{equation}\label{generalbasisQ0}
C_{i,j}=\partial_{A_{i,j}}+\sum C^{0,\lambda}_{i,j,h}\partial_{v_{0,\lambda}^{h}}
\end{equation}
with unique coefficients $C^{0,\lambda}_{i,j,h}$ which may be thought of as fiber coordinates $v_{0,\lambda,i,j}^{h}$ in the first jet bundle of $\Dir$ corresponding to partial derivatives $\partial_{A_{i,j}}v_{0,\lambda}^{h}$. By lemma \ref{lem:conditions on coefficients of Qphi}, $Q$ is of the form $Q_{\phi_{1}}$ iff the coefficients $C^{0,\lambda}_{i,j,h}$ satisfy
\begin{equation}\label{conditionC1'0}
C^{0,\lambda}_{i,j,h} = C^{0,\lambda'}_{i,j',h} %\quad \text{whenever} \quad \delta+1_{i}=\delta'+1_{i'} \quad \text{and} \quad \lambda-1_{j}=\lambda'-1_{j'}
\end{equation}
whenever the indices satisfy
\begin{align}
\lambda_{j}>0,& \; \lambda'_{j'}>0 \label{conditionC1a' 0}\\ 
 \lambda-1_{j}&=\lambda'-1_{j'} \label{conditionC1c' 0}
\end{align}
and condition
\begin{equation}
C^{0,\lambda}_{i,j,h} = 0 \quad \text{whenever} \quad \lambda_{j}=0. \label{conditionC2' 0}
\end{equation}
These are precisely the pasting conditions \ref{infpaste1} and \ref{infpaste1}.
\end{proof}

We finish the proof of the main theorem with
\begin{lem}
When $l>1$ The only maximal integral elements of  $(\Iln,\Dub)$ transversal to $\Iln\to \Il$ are the vertical tangent spaces of the projection 
\begin{equation}
\pr_{n}:\Gr(\uR,l)\times \Jkp_{\thetak} \to  \Jkp_{\thetak}
\end{equation}
i.e. planes of the distribution $\F^{-1}$. So the maximal integral submanifolds of $\Dub$ are the fibers of $\Iln \to \Jkp_{\thetak}$ and hence correspond bijectively to \virg{full} jets of order $k+1$ extending $\thetak$. This proves that the prolongation of $(\Iln,\Dub)$ is $(\Iln,V\pr_{n})$ which is an involutive distribution.
 \end{lem}
 
 \begin{proof}
Follows directly from lemma \ref{lem:coefficents of integral Q's} equation \ref{conditionC2} since in this case $\lambda=0$.
\end{proof}

This concludes the proof of main theorem \ref{thm:main}.
 
\section{Notational conventions}\label{sec:conventions}
For a finite dimensional vector space $W$ over a field $\mathbb{K}$, and $V\subset W$ a subspace we use the following conventions:
\begin{enumerate}
\item $\Gr(W,l)$ denotes the Grassmannian of all $l$ dimensional subspaces of $W$.
\item $S^{k}W$ denotes the $k\Th$ symmetric tensor product of $W$.
\item $W^{*}$ denotes the dual $\hom(W,\mathbb{K})$.
\item $V^{\circ}\subset W^{*}$ denotes the annihilator of $V$.
\item $W/V$ denotes the quotient. 
\item $\langle S \rangle$ denotes the span of the subset $S\subset W$
\end{enumerate}
For manifolds $M,N$ and a map $f:M\to N$ we use the conventions:
\begin{enumerate}
\item $Tf:TM\to TN$ denotes the tangent map.
\item $f^{-1}(S)$ denotes the preimage of subset $S\subset N$ under $f:M\to N$.
\item $M_q:=f^{-1}(\{q\})$ denotes the fiber over $q\in N$ when $f:M\to N$ is a bundle. 
\item An $f$-horizontal plane is a tangent subspace of $M$ transversal to the fibers of $f$.
\item $Vf$ denotes the vertical distribution of $f$ when it is a fiber bundle.
\item For a chart $x_{1},\ldots, x_{n}$ on $N$ the associated coordinate fields are denoted with $\partial_{x_{i}}$
\end{enumerate}
For a multinidex $\delta=(\delta_{1},\ldots,\delta_{n})\in\mathbb{N}^{n}$ and variables $x_{1},\ldots,x_{n}$:
\begin{enumerate}
\item $x^{\delta}=x_{1}^{\delta_{1}}\cdot \ldots \cdot x_{n}^{\delta_{n}}$.
\item $\delta! = \delta_{1}!\cdot\ldots\cdot\delta_{n}!$ is the factorial of the multiindex.
\item $|\delta|=\delta_{1}+\ldots+\delta_{n}$ denotes the length of the multiindex.
\item $1_{j}$ denotes the mutliindex with all zero entries except for the entry at position $j$ equaling $1$.
\end{enumerate}

\subsubsection*{Acknowledgements}
I thank Alexandre Vinogradov for drawing my attention to the polar distribution and encouraging me to study it. I thank Giovanni Moreno for useful conversations as well as an invitation to the University of Opava in the summer of 2013, where part of this work was completed. Finally I thank the anonymous reviewer for many useful suggestions on improving the presentation.\par
%
%
%\bibliographystyle{plain}
%\addcontentsline{toc}{section}{\refname}

\end{document}